\documentclass{amsart}

\newtheorem{theorem}{Theorem}[section]
\newtheorem{proposition}[theorem]{Proposition}
\newtheorem{lemma}[theorem]{Lemma}
\newtheorem{corollary}[theorem]{Corollary}
\newtheorem{fact}[theorem]{Fact}

\theoremstyle{definition}
\newtheorem{definition}[theorem]{Definition}

\theoremstyle{remark}
\newtheorem{remark}[theorem]{Remark}
\newtheorem{example}[theorem]{Example}

\def\Ind{\setbox0=\hbox{$x$}\kern\wd0\hbox to 0pt{\hss$\mid$\hss} \lower.9\ht0\hbox to 0pt{\hss$\smile$\hss}\kern\wd0} 

\def\Notind{\setbox0=\hbox{$x$}\kern\wd0\hbox to 0pt{\mathchardef \nn=12854\hss$\nn$\kern1.4\wd0\hss}\hbox to 0pt{\hss$\mid$\hss}\lower.9\ht0 \hbox to 0pt{\hss$\smile$\hss}\kern\wd0} 

\def\ind{\mathop{\mathpalette\Ind{}}}

\newcommand{\compcent}[1]{\vcenter{\hbox{$#1\circ$}}} 

\newcommand{\comp}{\mathbin{\mathchoice 
{\compcent\scriptstyle}{\compcent\scriptstyle} 
{\compcent\scriptscriptstyle}{\compcent\scriptscriptstyle}}} 

\numberwithin{equation}{section}

\def \d {\delta}

\def \dd {\partial}
\def \D {\Delta}
\def \t {\theta}
\def \T {\Theta}
\def \l {\langle}
\def \r {\rangle}
\def \L {\Lambda}
\def \I {\mathcal I}
\def \P {\mathcal P}
\def \NN {\mathbb N}
\def \ZZ {\mathbb Z}

\def \U {\mathbb U}

\def \la {\langle}
\def \ra {\rangle}
\def \al { \alpha}
\def \DD {\mathcal D}
\def \ta {\tau_{\DD/\D}}
\def \ld {\ell_s}
\def \G {\mathcal G}

\def \et {d_{D/\D}}

\title[Relative D-groups and differential Galois theory]{Relative D-groups and differential Galois theory in several derivations}
\author[Omar Le\'on S\'anchez]{Omar Le\'on S\'anchez \\
University of Waterloo \\
July 2013}
\date{}

\begin{document}

\begin{abstract}
The Galois theory of logarithmic differential equations with respect to \emph{relative D-groups} in partial differential-algebraic geometry is developed. This theory generalizes simultaneusly the parametrized Picard-Vessiot theory of Cassidy and Singer and the finite-dimensional theory of Pillay's generalized strongly normal extensions.
\end{abstract}

\maketitle

\begin{center}
\end{center}

\begin{center}
\end{center}

\section*{Introduction}

In \cite{Ko} Kolchin gives a thorough exposition of the fundamental results of the differential Galois theory for strongly normal extensions in partial differential fields of characteristic zero. This theory extended the Picard-Vessiot theory~(\cite{Ko3}, \cite{VS}), by enlarging the class of differential field extensions with \emph{good} Galois groups of automorphisms. He showed that the Galois group of a strongly normal extension can be identified with an algebraic group in the (total) constants and that this group is not necessarily a linear algebraic group (in contrast with the Picard-Vessiot case). 

In Pillay's work \cite{Pi} the role of the constants in the definition of strongly normal extension was succesfully replaced by an arbitrary finite-dimensional definable set $X$ (i.e. a set of differential type zero), and there he presented the more general theory of $X$-strongly normal extensions (Kolchin's extensions being the case when $X$ is the field of constants). Pillay showed that the Galois group of an $X$-strongly normal extension is a finite-dimensional differential-algebraic group (not necessarily in the constants). It is worth pointing out that in \cite{Pi} Pillay actually works in the ordinary case (only one derivation); however, his results extend more or less immediately to the partial case as long as $X$ is finite-dimensional. One of the intentions of this paper is to extend Pillay's results to possibly infinite-dimensional $X$.

In \cite{Pi2}, Pillay observes that his generalized strongly normal extensions can be viewed as Galois extensions corresponding to \emph{logarithmic differential equations on D-groups}. Briefly, a D-group is an algebraic group $G$ equipped with a section $s:G\to \tau G$, where $\tau G$ is the (algebraic) prolongation of $G$ with respect to the derivations $\Pi=\{\d_1,\dots,\d_m\}$. There is a natural crossed-homomorphism $\ld :G\to \tau G_e$ called the \emph{logarithmic derivative} (see \S 5 below) associated to a D-group. The equations of which Pillay's strongly normal extensions are Galois extensions are of the form $\ld x=a$. When $G=GL_n$ and $s$ is the zero section, one recovers the classical Galois theory of linear differential equations (i.e. the Picard-Vessiot theory).

The theory of D-groups is very much part of finite-dimensional differential-algebraic geometry because the set of sharp points $$(G,s)^\#:=\{g\in G: s(g)=(g,\d_1 g,\dots,\d_m g)\}$$ is a finite-dimensional $\Pi$-algebraic group (working inside a model of $DCF_{0,m}$) and all finite-dimensional definable groups are, up to definable isomorphism, of this form.

To find an infinite-dimensional analogue of Pillay's differential Galois theory we introduce \emph{relative D-groups w.r.t. $\DD/\D$}, namely a $\D$-algebraic group $G$, where $\D\subset \Pi$, and a $\D$-section $s:G\to \tau_{\DD/\D}G$, where $\DD:=\Pi\setminus \D$ and $\tau_{\DD/\D}G$ is the \emph{relative prolongation (w.r.t. $\DD/\D$)} that was first introduced in \cite{Le}. We show that every definable group of differential type less than the number of derivations is the set of sharp points of a relative D-group in the above sense, up to definable isomorphism and possibly after replacing $\Pi$ by some independent linear combination.

We introduce logarithmic differential equations on relative D-groups, define the associated Galois extensions, and show that they correspond to the $X$-strongly normal extensions for arbitrary (i.e. not necessarily finite-dimensional) definable sets $X$. In the case $G=GL_n$ and $s$ the zero section, we recover the parametrized Picard-Vessiot theory considered by Cassidy and Singer \cite{PM}.

Another generalization of Kolchin's strongly normal extensions in this direction was carried out by Landesman in \cite{La} . His extensions are a special case of what we present here. Indeed, the differential Galois theory we present generalizes Landesman's in more or less the same way that Pillay's generalized Kolchin's.

The paper is organized as follows. In Section 1 we discuss differential type which will play a central role for us. In Section 2 we review/develop the fundamental results of generalized strongly normal extensions but in the possibly infinite-dimensional setting (i.e. arbitrary differential type). Then, in Sections 3 and 4, we develop the basic theory of relative D-varieties and relative D-groups, respectively. In Section 5, we reinterpret generalized strongly normal extensions as the Galois extensions associated to logarithmic differential equations on relative D-groups. Finally, in Section 6, we give two examples of non-linear Galois groups associated to logarithmic differential equations.

\

\noindent {\it Acknowledgements.} I would like to thank Rahim Moosa for all the useful discussions and support towards the completion of this article and Anand Pillay for his helpful comments on an earlier draft.

\section{Preliminaries on differential type}\label{prel}

We will work in the language of partial differential rings with $m$ derivations $\mathcal L_{m}=\{0,1,+,-,\times\}\cup\{\d_1,\dots,\d_m\}$, where $\Pi=\{\d_1,\dots,\d_m\}$ are the distinguished derivations. By $DCF_{0,m}$ we will mean the $\mathcal L_{m}$-theory of differentially closed fields of characteristic zero with $m$ commuting derivations. In \cite{Mc}, McGrail shows that $DCF_{0,m}$ is a complete theory that admits elimination of quantifiers and of imaginaries and it is $\omega$-stable. All these model-theoretic properties will be used throughout the paper.

Let us fix a sufficiently saturated $(\U,\Pi)\models DCF_{0,m}$ and a ground $\Pi$-field $K<\U$. We will identify all definable sets with their $\U$-points. Given any $A\subset \U$, we will denote by $K\{ A\}_{\Pi}$ and $K\l A\r_{\Pi}$ the $\Pi$-ring and $\Pi$-field generated by $A$ over $K$, respectively. The $\Pi$-constants of any $\Pi$-field $F\leq \U$ will be denoted by $F^{\Pi}=\{a\in F:\, \d a=0\text{ for all } \d\in \Pi\}$. 

We denote by $\T_\Pi$ the set of $\Pi$-derivatives; that is, $\T_\Pi=\{\d_m^{e_m}\cdots \d_1^{e_1}:\, e_i\geq 0\}$. By $\T_\Pi(h)$ we mean the set of derivatives of order less than or equal to $h\in \NN$. Recall that the $\Pi$-ring of $\Pi$-polynomials over $K$ in the differential indeterminates $x=(x_1,\dots,x_n)$ is just $K\{x\}_{\Pi}=K[\t x_i:\, i=1,\dots,n, \, \t\in \T_\Pi]$ and the $\Pi$-field of $\Pi$-rational functions over $K$, $K\l x\r_{\Pi}$, is just the field of fractions of $K\{x\}_{\Pi}$.

\begin{fact}[\cite{Ko}, Chap.2, \S12]\label{poli}
Let $a=(a_1,\dots,a_n)$ be a tuple from $\U$. There exists a numerical polynomial $\omega_{a/K}$ with the following properties:
\begin{enumerate}
\item For sufficiently large $h\in \NN$, $\omega_{a/K}(h)$ equals the transcendence degree of $K(\t a_i:\, i=1,\dots,n,\, \t\in \T_\Pi(h))$ over $K$.
\item deg $\omega_{a/K}\leq m$.
\item If we write $\omega_{a/K}=\sum_{i=0}^m d_i$$x+i\choose{i}$ where $d_i\in \ZZ$, then $d_m$ equals the $\Pi$-transcendence degree of $K\l a\r_{\Pi}$ over $K$.
\item If $b$ is another tuple such that $b\in K\l a\r_{\Pi}$, then there is $h_0\in \NN$ such that for sufficiently large $h\in \NN$ we have $\omega_{b/K}(h)\leq \omega_{a/K}(h+h_0)$.
\end{enumerate}
\end{fact}

The polynomial $\omega_{a/K}$ is called the \emph{Kolchin polynomial of} $a$ \emph{over} $K$. Even though the Kolchin polynomial is not in general a $\Pi$-birational invariant, (4) of Fact \ref{poli} shows that its degree is. We call this degree the $\Pi$\emph{-type of} $a$ \emph{over} $K$ and denote it by $\Pi$-type$(a/K)$. Similarly, if we write $\omega_{a/K}=\sum_{i=0}^m d_i$$x+i\choose{i}$ where $d_i\in \ZZ$, the coefficient $d_{\tau}$, where $\tau=\Pi$-type$(a/K)$, is a $\Pi$-birational invariant. We call $d_{\tau}$ the \emph{typical} $\Pi$\emph{-dimension} of $a$ \emph{over} $K$ and denote it by $\Pi$-dim$(a/K)$. We will adopt the convention that the degree and leading coefficient of the zero polynomial are both zero, so if $\omega_{a/K}=0$ then $\Pi$-type$(a/K)=0$ and $\Pi$-dim$(a/K)=0$.

As a remark, to those familiar with Laskar rank, let us mention that $RU(a/K)\geq \omega^m$ if and only if $\Pi$-type$(a/K)=m$ \footnote{This equivalence is well known, though, to the author's knowledge, a proof has not appeared anywhere.}. Also, we have that $\Pi$-type$(a/K)=0$ implies $RU(a/K)<\omega$ (it is not known if the converse holds).

A key theorem regarding $\Pi$-type and typical $\Pi$-dimension is the following:

\begin{fact}[\cite{Ko}, Chap. 2, \S 13]\label{Kolteo}
Let $a$ be a tuple from $\U$. Then there is a set $\D$ of linearly independent elements of the $K^{\Pi}$-vector space $\operatorname{span}_{K^\Pi}\Pi$ with $|\D|=\Pi$-type$(a/K)$, and a finite tuple $\al$ from $\U$, such that
\begin{displaymath}
K\la a\ra_{\Pi}=K\la\al\ra_{\D}
\end{displaymath}
with $\D$-type$(\al/K)=\Pi$-type$(a/K)$ and $\D$-dim$(\al/K)=\Pi$-dim$(a/K)$.
\end{fact}

We introduce the following terminology for the sake of convenience.

\begin{definition}[\it{Bounding and witnessing $\Pi$-type}]
Let $a$ be a finite tuple from $\U$, a set $\D$ of linearly independent elements of the $K^{\Pi}$-vector space $\operatorname{span}_{K^\Pi}\Pi$ is said to \emph{bound the $\Pi$-type of $a$ over $K$} if the $\Pi$-field $K\l a\r_{\Pi}$ is finitely $\D$-generated over $K$. Moreover, if $\D$ bounds the $\Pi$-type of $a$ over $K$ and $|\D|=\Pi$-type$(a/K)$ then we say that $\D$ \emph{witnesses the $\Pi$-type of $a$ over $K$}. 
\end{definition}

Fact \ref{Kolteo} says that given any finite tuple $a$ we can always find $\D$ witnessing the $\Pi$-type of $a$ over $K$.

\begin{lemma}\label{onb}
Let $a$ be a tuple from $\U$, $L$ a $\Pi$-field extension of $K$ and $\D$ a set of linearly independent elements of $\operatorname{span}_{K^\Pi}\Pi$. If $a\ind_K L$, then $\D$ bounds the $\Pi$-type$(a/L)$ if and only if it bounds the $\Pi$-type$(a/K)$.
\end{lemma}
\begin{proof}
Assume $\D$ bounds the $\Pi$-type$(a/L)$. Then we can find $\alpha$ a tuple of the form $(a,\partial_1 a,\dots,\partial_s a:\, \partial_i\in \T_\Pi, s<\omega)$ such that $L\l a\r_{\Pi}=L\l \alpha \r_{\D}$; here if $a=(a_1,\dots,a_n)$ and $\partial \in \T_\Pi$ then $\partial a=(\partial a_1,\dots,\partial a_n)$. It suffices to prove the following: 

\vspace{.07in}

\noindent {\bf Claim.} $K\l a\r_{\Pi}=K\l \alpha,\d_{1}\alpha,\dots,\d_m\alpha\r_{\D}$, where $\Pi=\{\d_1,\dots,\d_m\}$.

\vspace{.05in}

\noindent Let $w\in K\l \alpha,\d_{1}\alpha,\dots,\d_m\alpha\r_{\D}$, we need to show that $\d w\in K\l \alpha,\d_{1}\alpha,\dots,\d_m\alpha\r_{\D}$ for all $\d\in \Pi$. Since $w\in L\l \alpha \r_{\Pi}=L\l \alpha\r_{\D}$, then $w=f(\alpha)$ for some $\D$-rational function $f$ over $L$. Thus, there is a tuple $\beta$ of the form $(\al,\partial_1'\al,\dots,\partial_r'\al:\, \partial_i'\in \T_\D, r< \omega)$ such that $w=g(\beta)$ for some rational function $g$ over $L$. Let $\rho$ be a tuple whose entries form a maximal algebraically independent (over $K$) subset of the set consisting of the entries of $\beta$. Then $(w,\rho)$ is algebraically dependent over $L$. But since $a\ind_K L$, $K\l a\r_\Pi$ is algebraically disjoint from $L$ over $K$. Thus, since $(w,\rho)$ is from $K\l a\r_{\Pi}$, we get that $(w,\rho)$ is algebraically dependent over $K$. Thus, there is $h\in K[x,y]$ such that $h(w,\rho)=0$. Moreover, since $\rho$ is algebraically independent over $K$, $w$ appears non trivially in $h(w,\rho)$. Hence, $w$ is algebraic over $K(\rho)$ and so $\d w\in K(w,\rho,\d\rho)\subseteq K\l \alpha,\d_{1}\alpha,\dots,\d_m\alpha\r_{\D}$.
\end{proof}

Given a complete type $p$ over $K$ its Kolchin polynomial $\omega_p$ is defined to be $\omega_{a/K}$ for any $a$ realising $p$. Thus it makes sense to talk about the $\Pi$-type and typical $\Pi$-dimension of complete types. Now we extend these concepts to definable sets. First recall that we can put a total ordering on numerical polynomials by eventual domination, i.e. $f\leq g$ if and only if $f(h)\leq g(h)$ for all sufficiently large $h\in \NN$.

\begin{definition}
Let $X$ be a definable set. The Kolchin polynomial of $X$ is defined by
\begin{displaymath}
\omega_X=\max\left\{ \omega_{a/F}:\, a\in X\right\},
\end{displaymath}
where $F$ is any $\Pi$-field over which $X$ is defined (the fact that this does not depend on $F$ is a consequence of Theorem~4.3.10 of \cite{Mc}). In Lemma \ref{max} below we will see that $\omega_X=\omega_{a/F}$ for some $a\in X$. We define the $\Pi$-type of $X$ and its typical $\Pi$-dimension in the obvious way. Also, $X$ is said to be \emph{finite-dimensional} if $\Pi$-type$(X)=0$.
\end{definition}

Let $V$ be an irreducible affine $\Pi$-algebraic variety defined over $K$. We say that $a$ is a $\Pi$\emph{-generic point of $V$ over $K$} if $a\in V$ but $a$ is not in any proper $\Pi$-algebraic subvariety of $V$ defined over $K$. The \emph{generic type of $V$ over $K$} is the complete type over $K$ of any $\Pi$-generic point of $V$ over $K$.

\begin{lemma}\label{max}
Let $X$ be a $K$-definable set and let $V_1, \dots, V_s$ be the $K$-irreducible components of the $\Pi$-closure of $X$ over $K$ (in the $\Pi$-Zariski topology). If $p_i$ is the generic type of $V_i$ over $K$, then
\begin{displaymath}
\omega_X=\max\{\omega_{p_i}:\, i=1,\dots,s\}.
\end{displaymath}
\end{lemma}
\begin{proof}
Let $b\in X$, then $b$ is in some $V_i$. Let $a_i$ be a realisation of $p_i$, then there is a $\Pi$-ring homomorphism $f_i:K\{a_i\}_{\Pi}\to K\{b\}_{\Pi}$ over $K$ such that $a_i\mapsto b$. Thus, for any $h\in \NN$ the transcendence degree of $K(\t a_{i,j}: j=1,\dots,n,\, \t\in \T(h))$ over $K$, where $a_i=(a_{i,1},\dots,a_{i,n})$, is greater than or equal to the transcendence degree of $K(\t b_j: j=1,\dots,n',\, \t\in \T(h))$ over $K$, where $b=(b_1,\dots,b_{n'})$. This implies that, for sufficiently large $h\in \NN$, $\omega_{a_i/K}(h)\geq\omega_{b/K}(h)$ and hence $\omega_{a_i/K}\geq\omega_{b/K}$.
\end{proof}

\begin{definition}
Let $X$ be a $K$-definable set. A set $\D$ of linearly independent elements of $\operatorname{span}_{K^\Pi}\Pi$ is said to \emph{bound the $\Pi$-type of $X$} if $\D$~bounds the $\Pi$-type$(a/K)$ for all $a\in X$. We say $\D$ \emph{witnesses the $\Pi$-type of $X$} if $\D$ bounds the $\Pi$-type of $X$ and $|\D|=\Pi$-type$(X)$.
\end{definition}

By Lemma \ref{onb}, this definition does not depend on the choice of $K$. In other words, $\D$ bounds the $\Pi$-type$(X)$ if and only if there is a $\Pi$-field $F$, over which $X$ and $\D$ are defined, such that $\D$ bounds the $\Pi$-type$(a/F)$ for all $a\in X$. Indeed, suppose that $\D$-bounds the $\Pi$-type of $X$, $F$ is a $\Pi$-field over which $X$ and $\D$ are defined, and $a\in X$. One can assume, without loss of generality, that $F<K$. Let $b$ be a tuple from $\U$ such that $b\ind_F K$ and $tp(b/F)=tp(a/F)$. By assumption, $\D$ bounds the $\Pi$-type$(b/K)$, and so, by Lemma~\ref{onb}, $\D$ bounds the $\Pi$-type$(b/F)$. Hence, $\D$ bounds $\Pi$-type$(a/F)$.

\begin{remark}\label{witg}
It is not clear (at least to the author) that every definable set has a witness to its $\Pi$-type. However, for definable groups one can always find such a witness. To see this let $G$ be a $K$-definable group. Note that, since the property ``$\D$ witnesses the $\Pi$-type'' for definable sets is preserved under definable bijection, a witness to the $\Pi$-type of the connected component of $G$ will be a witness to the $\Pi$-type of $G$. Hence, we may assume that $G$ is connected. Let $p$ be the generic type of $G$ over $K$. By Fact \ref{Kolteo}, there is $\D$, a linearly independent subset of $\operatorname{span}_{K^\Pi}\Pi$, such that for any $a\models p$ the $\Pi$-field $K\l a\r_{\Pi}$ is finitely $\D$-generated over $K$. Let us check that this $\D$ witnesses the $\Pi$-type of $G$. Let $g\in G$, then there are $a,b\models p$ such that $g=a\cdot b$ (see \S 7.2 of \cite{Ma2}). Thus $K\l g\r_{\Pi}\leq K\l a,b\r_{\Pi}$, but since the latter is finitely $\D$-generated over $K$ the former is as well.
\end{remark}

\begin{lemma}\label{onint}
Let $X$ be a $K$-definable set and $a$ a tuple from $\U$. Suppose $tp(a/K)$ is $X$-internal, then
\begin{enumerate}
\item $\Pi\text{-type}(a/K)\leq \Pi\text{-type}(X)$.
\item If $\D$ bounds the $\Pi$-type$(X)$ then $\D$ also bounds the $\Pi$-type$(a/K)$.
\end{enumerate}
\end{lemma}
\begin{proof} Recall that internality means that there is $L\geq K$ with $a\ind_K L$ and a finite tuple $c$ from $X$ such that $a\in\operatorname{dcl}(L,c)=L\l c\r_\Pi$.

\noindent (1)  For any $d=(d_1,\dots,d_r)$, $\Pi$-type$\displaystyle(d/K)=\max_{i\leq r}\{\Pi$-type$(d_i/K)\}$ (see for example Lemma 3.1 of \cite{Moo}). Hence,
\begin{displaymath}
\Pi\text{-type}(a/K)=\Pi\text{-type}(a/L)\leq \Pi\text{-type}(c/L)\leq \Pi\text{-type}(X),
\end{displaymath}
where the equality follows from Lemma \ref{onb}.

\noindent (2) Assume $\D$ bounds the $\Pi$-type$(X)$, then $L\l c\r_{\Pi}$ is finitely $\D$-generated over $L$. Since $L\l a\r_{\Pi}\leq L \l c\r_{\Pi}$, then $L\l a\r_{\Pi}$ is also finitely $\D$-generated over $L$. But $a\ind_K L$, so, by Lemma \ref{onb}, $\D$ bounds the $\Pi$-type of $a$ over $K$.
\end{proof}

\

\section{Generalized strongly normal extensions}\label{genstrong}

In this section we extend Pillay's strongly normal extensions \cite{Pi} to the possibly infinite-dimensional partial differential setting. In fact most of Pillay's arguments extend to this setting. We therefore focus on the statements rather than the proofs, giving some details only for the sake of review and to fix notation.

Fix a sufficiently saturated $(\U,\Pi)\models DCF_{0,m}$ and a base $\Pi$-field $K < \U$. All definable sets will be identified with their $\U$-points. Unless otherwise specified the notation and terminology of this section will be with respect to the language of $\Pi$-fields; for example, $K\l x\r$ means $K\l x\r_{\Pi}$, generated means $\Pi$-generated, isomorphism means $\Pi$-isomorphism, etc.

\begin{definition}\label{def}
Let $X$ be $K$-definable. A finitely generated $\Pi$-field extension $L$ of $K$ is said to be $X$-\emph{strongly normal} if:
\begin{enumerate}
\item $X(K)=X(\bar L)$ for some (equivalently any) differential closure $\bar L$ of $L$.
\item For any isomorphism $\sigma$ from $L$ into $\U$ over $K$, $\sigma(L)\subseteq L\l X\r$.
\end{enumerate}
\end{definition}

A $\Pi$-field extension $L$ of $K$ is called a \emph{generalized strongly normal extension} of $K$ if it is an $X$-strongly normal extension for some $K$-definable $X$. 

\begin{remark} \label{rem} \
\begin{enumerate} 
\item [(i)] Suppose $\DD=\{\d_1,\dots,\d_r\}$, $\D=\{\d_{r+1},\dots,\d_m\}$ is a partition of the derivations $\Pi=\{\d_1,\dots,\d_m\}$ with $0<r\leq m$. If $X=\U^{\DD}$, then the $X$-strongly normal extensions are exactly the ``$\D$-strongly normal extensions" of Landesman \cite{La}. Indeed, one need only observe that in this case~(1) is equivalent to $K^\DD=L^{\DD}$ is $\D$-closed (see for example Lemma 9.3 of \cite{PM}). 
\item [(ii)] A compactness argument, just as in Lemma 2.5 of \cite{Pi}, , shows that if $L$ is a generalized strongly normal extension of $K$ then $L$ is contained in some  differential closure of $K$. Therefore, condition~(1) can be replaced by 
\begin{displaymath}
  \qquad X(K)=X(\bar K) \text{ and } L\subseteq \bar K \text{ for some differential closure } \bar K \text{ of } K.
\end{displaymath}
\item [(iii)] If $L=K\l a\r$ is an $X$-strongly normal extension and $b$ realizes $tp(a/K)$, then $L'=K\l b\r$ is also an $X$-strongly normal extension. Moreover, if $b\in \bar L$ then $L=L'$. Indeed, if $\sigma$ is an automorphism of $\U$ over $K$ such that $\sigma(a)=b\in \bar L$, then, by (2), $L'\subseteq L\l X\r$. But as $L'\subseteq \bar L\models DCF_{0,m}$, we have that $L'\subseteq L\l X(\bar L)\r=L\l X(K)\r=L$. A similar argument shows $L\subseteq L'$.
\item [(iv)] If $L$ is an $X$-strongly normal extension of $K$ then, as in Remark 2.6 of \cite{Pi}, every isomorphism from $L$ into $\U$ over $K$ extends uniquely to an automorphism of $L\l X\r$ over $K\l X\r$. 
\item [(v)] When $X$ is finite, $X$-strongly normal extensions are just the usual Galois (i.e. finite and normal) extensions of $K$. On the other hand, if $\Pi$-type$(X)=m$ and $X(\bar K)=X(K)$, then $\bar K=K$ (this is because $X$ having $\Pi$-type equal to $m$ implies that $X$ projects $\Pi$-dominantly onto some coordinate). Hence, if $\Pi$-type$(X)=m$, $X$-strongly normal extensions are trivial. Thus, this notion is mainly of interest when $X$ is infinite and of $\Pi$-type strictly less than $m$.
\end{enumerate}
\end{remark}

For the rest of this section we fix a $K$-definable set $X$.

\begin{proposition}\label{inter}
For any $L$ finitely generated $\Pi$-field extension of $K$, condition~(2) of Definition~\ref{def}  is equivalent to $L$ being generated by a fundamental system of solutions of an $X$-internal type over $K$. 
\end{proposition}
\begin{proof}
Recall that if a type $p\in S_n(K)$ is $X$-internal there is a sequence $(a_1,\dots,a_\ell)$ of realisations of $p$, called a fundamental system of solutions, such that for every $b\models p$ there is a tuple $c$ from $X$ such that $b\in\operatorname{dcl}(K,a_1,\dots,a_\ell,c)$. 

Suppose $L=K\l a\r$ satisfies condition~(2), we claim that $tp(a/K)$ is $X$-internal and $a$ is itself a fundamental system of solutions. If $b\models tp(a/K)$ there is an automorphism $\sigma$ of $\U$ over $K$ such that $b=\sigma(a)\in L\l X\r$. Hence, $b\in \operatorname{dcl}(K,a,c)$ for some tuple $c$ from $X$. Conversely, suppose $L=K\l a\r$ and $a$ is a fundamental system of solutions of an $X$-internal type over $K$. Note that $tp(a/K)$ is also $X$-internal with fundamental system of solutions $a$ itself. So if $\sigma$ is an isomorphism from $L$ into $\U$ over $K$, then $\sigma(a)\models tp(a/K)$ and hence there is a tuple $c$ from $X$ such that $\sigma(a)\in \operatorname{dcl}(K,a,c)=L\l c\r$. That is $\sigma(L)\subseteq L\l X\r$, as desired.
\end{proof}

If $L=K\l a\r$ is an $X$-strongly normal extension then, by Lemma \ref{onint} and Proposition \ref{inter},
\begin{displaymath}
\Pi\text{-type}(a/K)\leq \Pi\text{-type}(X).
\end{displaymath}
Moreover, if $\D$ bounds the $\Pi$-type of $X$ then $\D$ also bounds the $\Pi$-type of $a$ over $K$. Specializing to the case when $\DD=\{\d_1,\dots,\d_r\}$, $\D=\{\d_{r+1},\dots,\d_m\}$ is a partition of $\Pi$ and $X=\U^{\DD}$, we recover Theorem 1.24 of \cite{La} that every $\U^\DD$-strongly normal extension of $K$ is finitely $\D$-generated over $K$.

Now, let $L=K\l a\r$ be an $X$-strongly normal extension of $K$. Following Pillay (and Kolchin), we define Gal$_X(L/K):=Aut_{\Pi}(L\l X\r/K\l X\r)$; that is, Gal$_X(L/K)$ is the group of automorphisms of $L\l X\r$ over $K\l X\r$. Note that $tp(a/K)^\U$ is in $\operatorname{dcl}(L\cup X)$, and thus, by (iv) of Remark \ref{rem}, the natural action of Gal$_X(L/K)$ on $tp(a/K)^\U$ is regular. We also let gal$(L/K):=Aut_{\Pi}(L/K)$, by (iv) of Remark \ref{rem}, every element of gal$(L/K)$ extends uniquely to an element of Gal$_X(L/K)$. Via this map, we identify gal$(L/K)$ with a subgroup of Gal$_X(L/K)$.

\begin{remark}
If $L$ is also an $X'$-strongly normal extension of $K$ then, by (iv) of Remark~\ref{rem}, there is a (unique) group isomorphism $$\psi:\, \text{Gal}_X(L/K)\to \,\text{Gal}_{X'}(L/K)$$ such that $\sigma|_{L\l X\r\cap L\l X'\r}=\psi(\sigma)|_{L\l X\r\cap L\l X'\r}$. Moreover, this isomorphism preserves the action of these groups on $tp(a/K)^\U$. We therefore often omit the subscript and simply write Gal$(L/K)$.
\end{remark}

The following theorem asserts the existence of the definable Galois group of a generalized strongly normal extension and summarises some of its basic properties.

\begin{theorem}\label{group}
Suppose $L=K\l a\r$ is an $X$-strongly normal extension of $K$.
\begin{enumerate}
\item There is a $K$-definable group $\G$ in $\operatorname{dcl}(K \cup X)$ with an $L$-definable regular action on $tp(a/K)^\U$ such that $\G$ together with its action on $tp(a/K)^\U$ is (abstractly) isomorphic to Gal$(L/K)$ together with its natural action on $tp(a/K)^\U$. We call $\G$ the Galois group of $L$ over $K$.
\item We have that $\Pi$-type$(\G)=\Pi$-type$(a/K)$ and $\Pi$-dim$(\G)=\Pi$-dim$(a/K)$. Moreover, $\D$ bounds the $\Pi$-type of $\G$ if and only if $\D$ bounds the $\Pi$-type of $a$ over $K$. 
\item $L$ is a $\G$-strongly normal extension.
\item The action of $\G$ on $tp(a/K)^\U$ is $K$-definable if and only if $\G$ is abelian.
\item $\G$ is connected if and only if $K$ is relatively algebraically closed in $L$.
\item If $\mu:$Gal$(L/K)\to \G$ is the isomorphism from (1), then $\mu($gal$(L/K))=\G(\bar K)=\G(K)$.
\end{enumerate}	
\end{theorem}
\begin{proof} \

\noindent (1) The construction is exactly as in \cite{Pi}, but we recall it briefly. Let $Z=tp(a/K)^{\U}$, by (ii)~of Remark \ref{rem}, $tp(a/K)$ is isolated and so $Z$ is $K$-definable. If $b\in Z$ then $b$ is a tuple from $L\l X\r$,  so by compactness we can find a $K$-definable function $f_0(x,y)$ such that for any $b\in Z$ there is a tuple $c$ from $X$ with $b=f_0(a,c)$. Let $Y_0=\{c \in X: f_0(a,c)\in Z\}$, then $Y_0$ is a $K$-definable set of tuples from $X$. Consider the equivalence relation on $Y_0$ given by $E(c_1,c_2)$ if and only if $f_0(a,c_1)=f_0(a,c_2)$, then $E$ is $K$-definable. By elimination of imaginaries we can find a $K$-definable set $Y$ in $\operatorname{dcl}(K\cup X)$ which we can identify with $Y_0/E$. Now define $f:Z\times Y\to Z$ by $f(b,d)=f_0(b,c)$ where $c$ is such that $d=c/E$. Now note that for any $b_1,b_2\in Z$ there is a unique $d\in Y$ such that $b_2=f(b_1,d)$, and so we can write $d=h(b_1,b_2)$ for some $K$-definable function $h$.

Define $\mu:$ Gal$_X(L/K)\to Y$ by $\mu(\sigma)=h(a,\sigma a)$. Then $\mu$ is a bijection. Let $\G$ denote the group with underlying set $Y$ and with the group structure induced by $\mu$. Then $\G$ is a $K$-definable group and is in $\operatorname{dcl}(K\cup X)$. Consider the action of $\G$ on $Z$ induced (via $\mu$) from the action of Gal$_X(L/K)$ on $Z$, i.e. for each $g\in \G$ and $b\in Z$ let $g.b:=\mu^{-1}(g)(b)$. This action is indeed $L$-definable since
\begin{equation}\label{action1}
\mu^{-1}(g)(b)=f(\mu^{-1}(g)(a),h(a,b))=f(f(a,g),h(a,b)).
\end{equation}

\noindent (2) Since $\G$ acts regularly and definably on $tp(a/K)^\U$, the map $g\mapsto g.a$ is a definable bijection between these two $K$-definable sets. Part (2) follows since $\Pi$-type, $\Pi$-dim and the property ``$\D$ bounds the $\Pi$-type" for definable sets are all preserved under extension of the set of parameters, as well as under definable bijection.

\noindent (3) Since $\G$ is in $\operatorname{dcl}(K\cup X)$ and $X(\bar K)=X(K)$, we have that $\G(\bar K)=\G(K)$. If $\sigma$ is an isomorphism from $L=K\l a\r$ into $\U$ over $K$, then, by Remark~\ref{rem}~(iv), we can extend it uniquely to an element of Gal$_X(L/K)$. Then, by (1), we get that $\sigma(a)=\mu (\sigma).a$ is in $\operatorname{dcl}( L\cup\G)$ and so $\sigma(L)\subseteq L\l \G\r$.
\vspace{.05in}

\noindent (4) If the action is $K$-definable, then for any $g_1,g_2\in \G$ we have
\begin{displaymath}
(g_1g_2).a=g_1.(\mu^{-1}(g_2)(a))=\mu^{-1}(g_2)(g_1.a)=(g_2g_1).a,
\end{displaymath}
so $g_1g_2=g_2g_1$. Conversely, suppose $\G$ is abelian and let $\sigma\in$ Gal$(L/K)$ be such that $\sigma(a)=b$, then
\begin{displaymath}
\mu^{-1}(g)(b)=\mu^{-1}(g)(\sigma(a))=\sigma(\mu^{-1}(g)(a))=\sigma(f(a,g))=f(b,g).
\end{displaymath}

Parts (5) and (6) follow from the construction of $\G$ just as in the ordinary case and we therefore refer the reader to Lemma 2.11 and Corollary 2.13 of \cite{Pi}.
\end{proof}

Now, even though the construction of $\G$ in Theorem \ref{group} depends on the choice of $(X,a,Y,f)$, if we choose different data $(X',a',Y',f')$, for the same extension $L$ over $K$, and construct the corresponding $\G'$ and $\mu'$, then $\G$ and $\G'$ will be $K$-definably isomorphic. Moreover, together with any $K$-definable bijection between $tp(a/K)^\U$ and $tp(a'/K)^\U$, this isomorphism gives rise to an isomorphism between the actions of $\G$ on $tp(a/K)^\U$ and $\G'$ on $tp(a'/K)^\U$. 

\begin{example}\label{Gg} 
Suppose we have a partition $\DD=\{\d_1,\dots,\d_r\}$, $\D=\{\d_{r+1},\dots,\d_m\}$ of $\Pi$, $X=\U^\DD$, $L$ is an $X$-strongly normal extension of $K$ and $\G$ is its Galois group. Then $\G$ is in $\operatorname{dcl}(K\cup \U^\DD)$. Hence, $\G$ is a $\D$-algebraic group over $K^\DD$ in the $\DD$-constants. Thus, (1) and (2) of Theorem \ref{group} recover Theorem 1.24 of \cite{La}.
\end{example}

We have the following Galois correspondence, which we state without proof as it follows precisely as in the ordinary case.
\begin{theorem}\label{normal}
Let $L=K\l a\r$ be an $X$-strongly normal extension of $K$ with Galois group $\G$, and let $\mu:$Gal$(L/K)\to \G$ be the isomorphism from Theorem \ref{group}. 
\begin{enumerate}
\item If $K\leq F \leq L$ is an intermediate $\Pi$-field, then $L$ is an $X$-strongly normal extension of $F$. Moreover,
\begin{displaymath}
\G_F:=\mu(\textrm{Gal}(L/F))
\end{displaymath}
is a $K$-definable subgroup of $\G$ and is the Galois group of $L$ over $F$. The map $F\mapsto\G_F$ establishes a 1-1 correspondence between the intermediate differential fields and $K$-definable subgroups of $\G$. 
\item $F$ is an $X$-strongly normal extension of $K$ if and only if $\G_F$ is a normal subgroup of $\G$, in which case $\G/\G_F$ is the Galois group of $F$ over $K$.
\end{enumerate}
\end{theorem}

Finally, we have a positive answer to the ``baby" inverse Galois problem. Again, Pillay's proof of Proposition 4.1 of \cite{Pi} extends to this setting.

\begin{proposition}\label{bab}
Let $G$ be a connected definable group with $\Pi$-type$(G)<m$. Then there is a $\Pi$-field $K$, over which $G$ is defined, and a $G$-strongly normal extension $L$ of $K$ such that $G$ is the Galois group of $L$ over $K$.
\end{proposition}

\

\section{Relative prolongations and relative D-varieties}

In ordinary differential Galois theory, Picard-Vessiot extensions are Galois extensions corresponding to certain differential equations on linear algebraic groups in the constants (i.e. linear ODE's). Pillay's generalized strongly normal extensions correspond to certain differential equations on D-groups. In this section and the next, we develop the ``relative" version of D-varieties and D-groups that will appear in understanding our generalized strongly normal extensions of the previous section as Galois extensions of certain differential equations. 

Fix a $\D$-field $K$ of characteristic zero, with derivations $\D=\{\d_1,\dots,\d_\ell\}$, and $D:K\to K$ a $\D$-derivation (i.e. $D$ is a derivation on $K$ commuting with $\D$). Recall that $\T_\D$ denotes the set of $\D$-derivatives; that is, $$\T_\D=\{\d_\ell^{e_{\ell}}\cdots \d_1^{e_1}:\, e_i\geq 0\}.$$ 

For each  $f\in K\{x\}_\D$ define $\et f \in K\{x,u\}_\D$ by
\begin{displaymath}
\et f (x,u):= \sum_{\t\in \T_{\D},\, i\leq n} \frac{\partial f}{\partial(\t x_i)}(x)\t u_i+f^{D}(x).
\end{displaymath}
Where $x=(x_1,\dots,x_{n})$ and $u=(u_1,\dots,u_{n})$, and the $\D$-polynomial $f^{D}$ is obtained by applying $D$ to the coefficients of $f$. Now extend $\et$ to all $\D$-rational functions over $K$ using the quotient rule; that is, if $f=\frac{g}{h}\in K\l x\r_\D$ then $$\et f(x,u)=\frac{\et g(x,u) \, h(x)-g(x)\,\et h(x,u)}{(h(x))^2}.$$ If $f=(f_1,\dots,f_s)$ is a sequence of $\D$-rational functions over $K$ by $\et f$ we mean $(\et f_1,\dots,\et f_s)$. We sometimes write $\et f_{x}(u)$ instead of $\et f(x,u)$. 

An easy computation shows that if $a$ is a tuple from $K$ then $$Df(a)=\et f(a,Da).$$

\begin{definition}
Let $V$ be an affine $\D$-algebraic variety defined over $K$ and $D$ a $\D$-derivation on $K$. The \emph{relative prolongation of $V$ (w.r.t. $D/\D$)}, denoted by $\tau_{D/\D} V$, is the affine $\D$-algebraic variety defined by
\begin{displaymath}
f(x)=0 \; \text{ and } \; \et f(x,y)=0,
\end{displaymath}
for all $f$ in $\I(V/K)_{\D}$. Here $\I(V/K)_\D$ denotes the $\D$-ideal of $V$ over $K$. If $a$ is a point in $V$, we denote the fibre of $\tau_{D/\D}V$ at $a$ by $\tau_{D/\D}V_a$.
\end{definition}

The notion of relative prolongation of an affine $\D$-algebraic variety was introduced in \cite{Le}, where it was used to formulate geometric first-order axioms for partial differentially closed fields.

\begin{remark}\label{remi} \
\begin{itemize}
\item [(i)] If $V$ is defined over $K^{D}$, the $D$-constants of $K$, then $\tau_{D/\D}V$ is just  $T_{\D} V$, the $\D$-tangent bundle of $V$ introduced by Kolchin in \cite{Ko2}. In general, $\tau_{D/\D}V$ is a $T_{\D} V$-torsor (see \cite{Le}).
\item [(ii)] Suppose the $K$-irreducible components of $V$ are $V_1,\dots,V_s$. If $a\in V_i\setminus \cup_{j\neq i}V_j$ then $\tau_{D/\D}V_a=\tau_{D/\D}(V_i)_a$ (see Proposition 4.1 of \cite{Le}).
\end{itemize}
\end{remark}

\begin{remark}
The relative prolongation can be described in terms of local $\D$-derivations along the lines of Kolchin's definition of the $\D$-tangent bundle. We give a brief explanation. Let $V$ be an affine $\D$-algebraic variety defined over $K$ with coordinate $\D$-ring $$K\{ V \}_{\D}:=K\{x\}_{\D}/\I\,,$$
where $x=(x_1,\dots,x_n)$. Let $a\in V(F)$, where $F$ is a $\D$-extension of $K$. The \emph{local} $\D$\emph{-ring of} $V$ \emph{at} $a$ \emph{over} $K$, $\mathcal{O}_a(V/K)_{\D}$, is just the localization of $K\{V\}_{\D}$ at the prime $\D$-ideal $\P:=\{f\in K\{V\}_{\D}:f(a)=0\}$. Note that $\mathcal{O}_a(V/K)_{\D}$ is a $\D$-ring extension of $K$. A \emph{local} $\D$\emph{-derivation at} $a$ is an additive map $\xi:\mathcal{O}_a(V/K)_\D\to F$ commuting with $\D$ such that $\xi(f g)=\xi(f)g(a)+f(a)\xi(f)$ for all $f,g\in \mathcal{O}_a(V/K)_\D$. If $\xi$ is a local $\D$-derivation at $a$ extending $D$, then $(\xi(\bar x_1),\dots,\xi(\bar x_n))$ is a solution in $F^n$ to the system $\{\et f_a(u)=0: f\in \I(V/K)_{\D}\}$, where $\bar x_i$ is the image in $\mathcal{O}_a(V/K)_\D$ of $x_i+\I$. Conversely, every solution $b$ in $F^n$ to the system $\{\et f_a(u)=0: f\in \I(V/K)_{\D}\}$ gives rise to a local $\D$-derivation at $a$ extending $D$ defined by
\begin{displaymath}
\xi_b (f)=\et f_a(b),
\end{displaymath}
and it satisfies $(\xi_b(\bar x_1),\dots,\xi_b(\bar x_n))=b$. Thus, the set of local $\D$-derivations at $a$ extending $D$ can be identified with the set of solutions in $F$ of the system $\{\et f_a(u)=0: f\in \I(V/K)_{\D}\}$. Hence, the $F$-points of $\tau_{D/\D}V_a$ can be identified with the set of all local $\D$-derivations at $a$ extending $D$. 
\end{remark}

\begin{example}\label{exem}
Let $\D=\{\d\}$, and consider the $\d$-algebraic variety $V\subseteq \mathbb{A}$ defined by 
$$x\neq 0\; \text{ and }\; \d\left(\frac{\d x}{x}\right)=0.$$ 
This is a standard example of a $\d$-algebraic subgroup of $\mathbb{G}_m$, see for example \cite{PM}. We wish to compute $\tau_{D/\d}V$. First note that $V$ is a quasi-affine $\d$-algebraic variety that can be identified with the affine $\d$-algebraic variety $W$ defined by $$f(x,y):=xy-1=0\; \text{ and }\; g(x,y):=x(\d^2 x)-(\d x)^2=0.$$
We get that 
$$d_{D/\d}f(x,y,u,v)=yu+xv \; \text{ and } \; d_{D/\d}g(x,y,u,v)=(\d^2x) u-2\d x \d u + x(\d^2 u).$$ By Rosenfeld's criterion (see \cite{Ko}, Chap. IV, \S 9), the $\d$-ideal generated by $f$ and $g$, denoted by $[f,g]$, is a prime $\d$-ideal of $K\{x,y\}_\d$. Thus, by the differential Nullstellensatz, $\I(W/K)_\d=[f,g]$, and so $\tau_{D/\d}W$ is defined by $f=g=d_{D/\d}f=d_{D/\d}g=0$. Therefore, the relative prolongation of $V$, $\tau_{D/\d}V \subseteq \mathbb{A}^2$, is given by $$x\neq 0, \quad \d\left(\frac{\d x}{x}\right)=0,\; \text { and }\; (\d^2x) u-2\d x \d u + x(\d^2 u)=0.$$
\end{example}

More generally, if $\DD=\{D_1,\dots,D_r\}$ is a set of commuting $\D$-derivations on $K$ then the \emph{relative prolongation of $V$ (w.r.t. $\DD/\D$)} is defined as the fibred product:
\begin{displaymath}
\tau_{\DD/\D} V:=\tau_{D_1/\D}V\times_V \cdots\times_V \tau_{D_r/\D}V.
\end{displaymath}

If $V$ and $W$ are affine $\D$-algebraic varieties and $f:V\to W$ is a regular $\D$-map (all defined over $K$). Then we have an induced regular $\D$-map, $\tau_{\DD/\D} f:\tau_{\DD/\D} V\to \tau_{\DD/\D}W$, given on points by
\begin{displaymath}
(x,u_1,\dots,u_s)\mapsto (f(x),d_{D_1/\D}f(x,u_1),\dots,d_{D_r/\D}f(x,u_r)).
\end{displaymath}

\begin{proposition}
$\tau_{\DD/\D}$ is a functor from the category of affine $\D$-algebraic varieties (defined over $K$) to itself. Moreover, $\tau_{\DD/\D}$ commutes with products; that is, $\tau_{\DD/\D}(V\times W)$ is naturally isomorphic to $\tau_{\DD/\D}V\times \tau_{\DD/\D}W$ via the map $$(x,y,u_1,\dots,u_r,v_1\dots,v_r)\mapsto(x,u_1,\dots,u_r,y,v_1,\dots,v_r).$$
\end{proposition}
\begin{proof}
For ease of notation we assume $\DD=\{D\}$. We need to show that for each pair of regular $\D$-maps $f:V\to W$ and $g:W\to U$ we have that $\tau_{D/\D}(g\comp f)=\tau_{D/\D}g\,\comp\,\tau_{D/\D}f$. By the definition of $\tau_{D/\D}$ it suffices to show that $\et(g\comp f)(x,u)=\et g(f(x),\et f(x,u))$. Write $f=(f_1,\dots,f_s)$. A straightforward but tedious computation shows that 
\begin{displaymath}
(g\comp f)^D(x)=\et g(f(x),f^D(x))=\sum_{\sigma\in \T_\D,\, j\leq s}\frac{\partial g}{\partial \sigma y_j}(f(x))\sigma(f^D_j)+g^D(f(x)).
\end{displaymath}
Also, for each $\sigma\in \T_\D$ and $j\leq s$ we have that
\begin{displaymath}
\sum_{\t\in \T_\D,\, i\leq n}\frac{\partial \sigma f_j}{\partial \t x_i}(x)\t u_i=\sum_{\t\in \T_\D,\, i\leq n} \sigma\left(\frac{\partial f_j}{\partial \t x_i}(x)\,\t u_i\right).
\end{displaymath}
Finally, using these equalities we have that
\begin{eqnarray*}
\et(g\comp f)
&=& \sum_{\sigma\in \T_\D,\, j\leq s}\;\sum_{\t\in \T_\D,\, i\leq n}\frac{\partial g}{\partial \sigma y_j}(f(x))\frac{\partial \sigma f_j}{\partial \t x_i}(x)\t u_i+(g\comp f)^D(x) \\
&=& \sum_{\sigma\in \T_\D,\, j\leq s}\frac{\partial g}{\partial \sigma y_j}(f(x))\left(\sum_{\t\in\T_\D,\,i\leq n}\frac{\partial \sigma f_j}{\partial \t x_i}(x)\t u_i+\sigma(f^D_j)\right)+ g^D(f(x)) \\
&=& \sum_{\sigma\in \T_\D,\, j\leq s}\frac{\partial g}{\partial \sigma y_j}(f(x))\sigma\left(\sum_{\t\in\T_\D,\,i\leq n}\frac{\partial f_j}{\partial \t x_i}(x)\t u_i +f^D_j\right)+g^D(f(x)) \\
\vspace{2in}
&=& \et g(f,\et f).
\end{eqnarray*}

The moreover clause follows from the fact that $\I(V\times W/K)_\D$ equals the ideal generated by $\I(V/K)_\D\subseteq K\{x\}_\D$ and $\I(W/K)_\D\subseteq K\{y\}_\D$ in $K\{x,y\}_\D$.
\end{proof}

Let $\DD=\{D_1,\dots,D_r\}$ be a set of commuting $\D$-derivations on $K$ and $\Pi=\DD\cup\D$. Fix a sufficiently saturated $(\U,\Pi)\models DCF_{0,m}$ such that $K < \U$, and identify all affine $\Pi$-algebraic varieties with their $\U$-points. Let $V$ be an affine $\D$-algebraic variety defined over $K$. Recall that for each $D\in \DD$, $f\in K\{x\}_\D$, and tuple $a$ (from $\U$) we have that $Df(a)=\et f(a,Da)$, and so for each $a\in V$ and $g\in \I(V/K)_\D$ we have $$\et g(a,Da)=Dg(a)=0.$$ Hence, $(a,D_1 a,\dots,D_r a)\in \tau_{\DD/\D}V$ for all $a\in V$. This motivates the following definition.

\begin{definition}
For each affine $\D$-algebraic variety $V$ defined over $K$, we define the map $\nabla_\DD^V$ from $V$ to $\tau_{\DD/\D}V$ by $$\nabla_{\DD}^Vx=(x,D_1 x,\dots,D_r x).$$
When $V$ is understood we simply write $\nabla_\DD$.
\end{definition}

\begin{remark}\label{commu}
If $f:V\to W$ is a regular $\D$-map between affine $\D$-algebraic varieties (all defined over $K$), then $\nabla_\DD^W\comp f=\tau_{\DD/\D}f\comp\nabla_\DD^V$. Indeed, since for each $D\in \DD$ and tuple $c$ we have that $Df(c)=\et f(c,Dc)$, for every $a\in V$  $$\nabla_\DD^W(f(a))=(a,\et f(\nabla_\DD^V a))=\tau_{\DD/\D}f( \nabla_\DD^V a).$$
\end{remark}

Now we introduce the relative version of D-varieties.

\begin{definition}\label{dvar}
Let $\DD=\{D_1,\dots,D_r\}$ be a set of commuting $\D$-derivations on $K$. An affine \emph{relative D-variety (w.r.t. $\DD/\D$)} defined over $K$ is a pair $(V,s)$ where $V$ is an affine $\D$-algebraic variety and $s$ is a $\D$-section of $\tau_{\DD/\D}V\to V$ (both defined over $K$) satisfying the following \emph{integrability condition}, for $i,j=1,\dots r$
\begin{equation}\label{rel1}
d_{D_i/\D}s_j(x, s_i(x))\equiv d_{D_j/\D}s_i(x, s_j(x)) \quad \text{mod} \; \I(V/K)_{\D},
\end{equation}
where $s=(\operatorname{Id},s_1,\dots,s_r)$. By a \emph{relative D-subvariety of} $(V,s)$ we mean a relative D-variety $(W,s_W)$ such that $W$ is a $\D$-subvariety of $V$ and $s_W=s|_W$.
\end{definition}

A morphism between affine relative D-varieties $(V,s)$ and $(W,t)$ is a regular $\D$-map $f:V\to W$ satisfying $\tau_{\DD/\D}f\comp s=t\comp f$. We define the category of affine $\DD/\D$-varieties in the obvious way.

\vspace{.08in}

Now let $(V,s)$ be an affine relative D-variety (w.r.t. $\DD/\D$) defined over $K$, we obtain a $\Pi$-algebraic variety (also defined over $K$) given by
\begin{displaymath}
(V,s)^{\sharp}=\{\, a\in V : \, s(a)=\nabla_{\DD}(a)\}.
\end{displaymath}
A point in $(V,s)^{\sharp}$ is called a \emph{sharp point} of $(V,s)$.

\begin{example}
Let $\DD=\{D\}$ and $\D=\{\d\}$. In example~\ref{exem} we saw that if $V<\mathbb{G}_m$ is the $\d$-algebraic subgroup defined by $x\neq 0$ and $\displaystyle \d\left(\frac{\d x}{x}\right)=0$, then $\tau_{D/\d}V\subseteq \mathbb{A}^2$ is given by
$$x\neq 0, \quad \d\left(\frac{\d x}{x}\right)=0,\; \text { and }\; (\d^2x) u-2\d x \d u + x(\d^2 u)=0.$$
An easy computation shows that if $\al\in K$ is such that $\d^2 \al=0$, then for all $a\in V$ we have that $(a,\al a)\in \tau_{D/\d}V$. Hence, we get a section $s:V\to \tau_{D/\d}V$ defined by $s(x)=(x,\al x)$, and so $(V,s)$ is a relative D-variety w.r.t. $D/\d$ defined over $K$. Suppose we are in the case when $K=F(w)$, where $F$ is a differential closure of $\left(\mathbb{C}(t),\frac{d}{dt}\right)$, and $D|_K=\frac{d}{dw}$ and $\d|_K=\frac{d}{dt}$, where we regard $t$ and $w$ as two complex variables. Let $\al=\frac{t}{w}$, then clearly $\d^2 \al=0$. In this case the sharp points of $(V,s)$ are given by $$(V,s)^\#=\left\{x\in V: \frac{dx}{dw}=\frac{t}{w} \,x\right\}.$$
The equation displayed above is a standard example of a linear differential equation parametrized by $t$. In \cite{PM}, it is shown that if $a\in (V,s)^\#$ then the Galois group of such an equation over $K$ is given by 
$$a^{-1}(V,s)^\#=\left\{x\in V: \frac{dx}{dw}=0\right\}.$$
\end{example}

\begin{proposition}\label{lema1} 
Let $(V,s)$ be an affine relative D-variety over $K$.
\begin{enumerate}  
\item $(V,s)^{\sharp}$ is $\D$-dense in $V$.
\item The $\sharp$ operation establishes a 1:1 correspondence between relative D-subvarieties of $(V,s)$ defined over $K$ and $\Pi$-algebraic subvarieties of $(V,s)^\#$ defined over $K$. The inverse is given by taking $\D$-closures (in the $\D$-Zariski topology) over $K$.
\item Assume $V$ is $K$-irreducible (as a $\D$-algebraic variety). There exists a $\D$-generic point $a$ of $V$ over $K$ such that $a\in (V,s)^{\sharp}$. Moreover, any such $a$ is a $\Pi$-generic point of $(V,s)^\#$ over $K$. In particular, $(V,s)^\#$ will be $K$-irreducible.
\item Assume $V$ is $K$-irreducible (as a $\D$-algebraic variety). Let $\omega_V$ be the Kolchin polynomial of $V$ as a $\D$-algebraic variety and let $\omega_{(V,s)^\#}$ be the Kolchin polynomial of $(V,s)^\#$ as a $\Pi$-algebraic variety. Write $$s=(\operatorname{Id},s_1,\dots,s_r),$$ and let $\mu=ord(s):=max\{ord(s_1),\dots,ord(s_r)\}$ in case $ord(s)\geq 1$, and $\mu=1$ in case $ord(s)=0$. Then for sufficiently large $h\in\NN$
\begin{displaymath}
\omega_{(V,s)^{\sharp}}(h)\leq \omega_{V}(\mu  h). 
\end{displaymath}
In particular, $\D$-type$(V)=\Pi$-type$(V,s)^{\sharp}$. 
\end{enumerate}
\end{proposition}

\begin{proof} \

\noindent (1) Let $F<\U$ be a $\Pi$-closed field extension of $K$ and $O$ be any (non-empty) $\D$-open subset of $V$ defined over $F$. We need to find a tuple $c$ from $F$ such that $c\in (V,s)^{\sharp}\cap O$. Let $W$ be an irreducible $\D$-component of $V$ such that $O\cap W\neq \emptyset$ and let $a$ be a $\D$-generic point of $W$ over $F$. Clearly $a\in O$. 

Now let $f\in \I(a/F)_{\D}=\I(W/F)_{\D}$. Since $a$ is a generic point of $W$ over $F$, we can find $g$ such that $g(a)\neq 0$ and $fg\in \I(V/F)_{\D}$. Thus, for $i=1,\dots,r$,
\begin{eqnarray*}
0
&=&d_{D_i/\D} (fg)_{a}(s_i(a)) \\
&=&d_{D_i/\D} f_{a}(s_i(a))g(a)+f(a)d_{D_i/\D} g_{a}(s_i(a)) \\
&=&d_{D_i/\D} f_{a}(s_i(a))g(a)
\end{eqnarray*}
since $f(a)=0$. But since $g(a)\neq 0$, we get $d_{D_i/\D}f_{a}(s_i(a))=0$ for $i=1,\dots,r$. This implies that there are $\D$-derivations $D_i':F\{a\}_{\D}\to F\{a,s_i(a)\}_{\D}=F\{a\}_{\D}$ extending $D_i:F\to F$ such that $D_i'(a)=s_i(a)$, for $i=1,\dots,r$ (see \cite{Ko2}, Chap. 0, \S4).

Let us check that $D_i'$ and $D_j'$ commute on $F\{a\}_{\D}$; that is, for $f\in F\{x\}_{\D}$, $[D_i',D_j']f(a)=0$. A rather lengthy computation shows that 
\begin{displaymath}
[D_i',D_j']f(a)=d_{[D_i,D_j]/\D}f(a,[D_i',D_j']a).
\end{displaymath}
Because $D_i$ and $D_j$ commute on $F$ we get $f^{[D_i,D_j]}=0$, so we only need to check that $[D_i',D_j']a=0$. Here we use the integrability condition (\ref{rel1}), we have
\begin{eqnarray*}
[D_i',D_j']a
&=& D_i' s_j(a)-D_j' s_i(a)\\
&=&d_{D_i/\D}s_{j}(a, D_i'a)-d_{D_j/\D}s_{i}(a,D_j'a)\\
&=& d_{D_i/\D}s_{j}(a, s_i(a))-d_{D_j/\D}s_{i}(a, s_j(a))=0.
\end{eqnarray*}
Thus, the $\D$-derivations $\DD'=\{D_{1}',\dots,D_r'\}$ commute on $F\{a\}_{\D}$. So, $F\{a\}_{\D}$ together with $\DD'$, is a differential ring extension of $(F,\Pi)$ and it has a tuple, namely $a$, living in $(V,s)^{\sharp}\cap O$. Hence, there is $c$ in $F$ such that $c\in (V,s)^{\sharp}\cap O$ as desired.

\noindent (2) If $(W,s_W)$ is a relative D-subvariety of $(V,s)$ defined over $K$ then, by (1), the $\D$-closure of $(W,s_W)^{\#}$ is just $W$. On the other hand, if $X$ is a $\Pi$-algebraic subvariety of $(V,s)^\#$ defined over $K$ and $W$ is its $\D$-closure over $K$, then $(W,s_W:=s|_W)$ is clearly a relative D-subvariety of $(V,s)$. We need to show $X=(W,s_W)^\#$. Towards a contradiction suppose $X\neq (W,s_W)^\#$. Using the equations of the sharp points $s(x)=\nabla_{\DD}x$, we obtain a $\D$-algebraic variety $U$ defined over $K$ such that $X\subseteq U$ and $(W,s_W)^\# \not\subseteq U$. This would imply that $(W,s_W)^\#\not\subseteq W$, but this is impossible.

\noindent (3) Consider the set of formulas
\begin{displaymath}
\Phi=\{x\in (V,s)^{\sharp}\}\cup\{x\notin W: W  \; \textrm{is a proper} \; \D\textrm{-algebraic subvariety of } V \textrm{ over } K \}
\end{displaymath}
if this set were inconsistent, by compactness and $K$-irreducibility of $V$, we would have that $(V,s)^{\sharp}$ is contained in a proper $\D$-algebraic subvariety of $V$, but this is impossible by part (1). Hence $\Phi$ is consistent and a realisation is the desired point. The rest is clear.

\noindent (4) Let $a$ be a $\D$-generic point of $V$ over $K$ such that $a\in (V,s)^{\sharp}$ (this is possible by part (3)). Then, since $\nabla_{\DD}a=s(a)$, we get that for each $h\in\NN$,
\begin{displaymath}
K(\dd a: \, \dd \in \T_{\Pi}(h))\subseteq K(\dd a: \, \dd\in \T_{\D}(\mu h)).
\end{displaymath} 
\end{proof}

\begin{remark}
So far in this section we have been using a partition $\{\DD,\D\}$ of $\Pi$, i.e. $\Pi=\DD\cup\D$ and $\DD\cap\D=\emptyset$. However, all the definitions and results make sense and hold if we replace such a partition for any two sets $\DD$ and $\D$ of linearly independent elements of the $K$-vector space $\operatorname{span}_{K^\Pi}\Pi$ such that their union $\DD\cup\D$ forms a basis.
\end{remark}

We finish this section by showing that each irreducible $\Pi$-algebraic variety defined over $K$ is $\Pi$-birationally equivalent to the sharp points of a relative D-variety w.r.t. $\DD/\D$ (with an appropriate choice of $\DD$ and $\D$ such that $\DD\cup\D$ is a basis of $\operatorname{span}_{K^\Pi}\Pi$). This is the analogue in several derivations of the well-known characterization of finite rank ordinary differential varieties (see \S 4 of \cite{PiT}).

\begin{proposition}\label{finteo}
Let $W$ be an affine $K$-irreducible $\Pi$-algebraic variety defined over $K$ with $\Pi$-type$(W)=l<|\Pi|$ and $\Pi$-dim$(W)=d$. Then there is a basis $\DD\cup \D$ of the $K^\Pi$-vector space $\operatorname{span}_{K^\Pi}\Pi$ with $|\D|=l$, and an affine relative D-variety $(V,s)$ w.r.t. $\DD/\D$ defined over $K$ with $\D$-type$(V)=l$ and $\D$-dim$(V)=d$, such that $W$ is $\Pi$-birationally equivalent (over $K$) to $(V,s)^{\sharp}$.
\end{proposition}

\begin{proof}
Let $a$ be a $\Pi$-generic point of $W$ over $K$. By Fact \ref{Kolteo}, we can find a set $\D$ of linearly independent elements of $\operatorname{span}_{K^\Pi}\Pi$ such that $K\la a\ra_{\Pi}=K\la\al\ra_{\D}$ for some tuple $\al$ of $\U$ with $\D$-tp$(\al/K)=l$ and $\D$-dim$(\al/K)=d$. Let $\DD=\{D_1,\dots,D_r\}$ be any set of derivations such that $\DD\cup\D$ is a basis of $\operatorname{span}_{K^\Pi}\Pi$.

For each $i=1,\dots,r$, we have
\begin{displaymath}
D_i \alpha=\frac{f_i(\al)}{g_i(\al)}
\end{displaymath}
where $\frac{f_i}{g_i}$ is a sequence of $\D$-rational functions over $K$. Let
\begin{displaymath}
b:= \left(\al,\frac{1}{g_{1}(\al)},\dots,\frac{1}{g_r(\al)}\right).
\end{displaymath}
Note that $b$ and $\al$ are $\D$-interdefinable over $K$, hence $\D$-type$(b/K)=\D$-type$(\al/K)$ and $\D$-dim$(b/K)=\D$-dim$(\al/K)$. Also, $a$ and $b$ are $\Pi$-interdefinable over $K$ and hence it suffices to show that $b$ is a $\Pi$-generic point of the set of sharp points of a relative D-variety w.r.t. $\DD/\D$ defined over $K$.

Let $V$ be the $\D$-locus of $b$ over $K$ (then $\D$-type$(V)=r$ and $\D$-dim$(V)=d$). A standard trick gives us a sequence $s=(\operatorname{Id},s_1,\dots,s_r)$ of $\D$-polynomials over $K$ such that $s(b)=\nabla_{\DD}b\in \tau_{\DD/\D}V$. Since $b$ is a $\D$-generic point of $V$ over $K$, we get that $(V,s)$ is a relative D-variety (w.r.t. $\DD/\D$) defined over $K$. Also, $b\in(V,s)^{\sharp}$ then, by Proposition \ref{lema1} (3), $b$ is a $\Pi$-generic point of $(V,s)^{\sharp}$ over $K$ as desired.
\end{proof}

The proof of Proposition \ref{finteo} actually proves:

\begin{corollary}\label{ong}
Let $W$ be an affine $K$-irreducible $\Pi$-algebraic variety defined over $K$ and $\D$ a set of linearly independent elements of $\operatorname{span}_{K^\Pi}\Pi$ such that $\D$ bounds the $\Pi$-type$(W)$ and $|\D|<|\Pi|$. If we extend $\D$ to a basis $\DD\cup\D$ of $\operatorname{span}_{K^\Pi}\Pi$, then $W$ is $\Pi$-birationally equivalent (over $K$) to $(V,s)^\#$ for some affine relative D-variety $(V,s)$ w.r.t. $\DD/\D$ defined over $K$.
\end{corollary}

\begin{remark}
While the definitions and results of this section where stated for affine $\D$-algebraic varieties (for the sake of concreteness), all of them make sense and hold for abstract $\D$-algebraic varieties. One only has to check that all the above constructions patch correctly at intersecting charts.
\end{remark}

\

\section{Relative D-groups}

In this section we develop the basic theory of relative D-groups. We prove the fact that every definable group is the set of sharp points of a relative D-group. We continue to work in our universal domain $(\U,\Pi)$ with a partition $\Pi=\DD\cup\D$, $\DD=\{D_1,\dots,D_r\}$, and over a base $\Pi$-field $K<\U$.

Let $\D'\subseteq \D$ and suposse $V$ is a $\D'$-algebraic variety defined over $K$. It is not known (at least to the author) if $\tau_{\DD/\D}V=\tau_{\DD/\D'}V$. However, this equality holds in the case of $\D'$-algebraic groups:

\begin{proposition}\label{prog}
Let $\D'\subseteq \D$ and let $G$ be a $\D'$-algebraic group defined over $K$. Then $\tau_{\DD/\D} G=\tau_{\DD/\D'}G$.
\end{proposition}
\begin{proof}
For ease of notation we assume $\DD=\{D\}$. By (ii) of Remark~\ref{remi} we may assume that $G$ is a connected $\D'$-algebraic group. Let $\L$ be a characteristic set of the prime $\D'$-ideal $\I(G/K)_{\D'}$ (we refer the reader to (\cite{Ko}, Chap. I, \S 10) for the definition of characteristic set). By Chapter 7 of \cite{Fre}, $G$ is also a connected $\D$-algebraic group and $\L$ is a characteristic set of $\I(G/K)_\D$. Let $a$ be a $\D$-generic point of $G$ over $K$. We claim that $\tau_{D/\D}G_a=\tau_{D/\D'}G_a$. Let $b\in \tau_{D/\D'}G_a$, then $\et f_a(b)=d_{D/\D'}f_a(b)=0$ for all $f\in \L$. It is easy to see now that $\et f_a(b)=0$ for all $f\in [\L]_\D$, where $[\L]_\D$ denotes the $\D$-ideal generated by $\L$ in $K\{x\}_\D$. Now let $g\in\I(G/K)_\D$, since $\L$ is a characteristic set of $\I(G/K)_\D$ we can find $\ell$ such that $H_\L^\ell\, g\in [\L]_\D$ , where $H_\L$ is the product of initials and separants of the elements of $\L$ (see Chap. IV, \S 9 of \cite{Ko}). Thus we have
\begin{eqnarray*}
0&=& \et \big(H_\L^\ell g\big)_a(b) \ \ \ \ \ \ \ \ \ \text{ as $H_\Lambda^\ell g\in[\L]_\D$}\\
&=& H_\L^\ell(a)\et g_a(b)+g(a)\et (H_\L^\ell)_a(b)\\
&=& H_\L^\ell(a)\et g_a(b).
\end{eqnarray*}
Since $a$ is a $\D$-generic point of $G$ over $K$, $H_\L(a)\neq 0$, and so $\et g_a(b)=0$. Hence, $\tau_{D/\D'}G_a\subseteq\tau_{D/\D}G_a$. The other containment is clear.

Now we show that $\tau_{D/\D}G_g=\tau_{D/\D'}G_g$ for all $g\in G$. Let $g\in G$, we can find $h\in G$ such that $g=ha$. Let $\lambda^h:G\to G$ denote left multiplication by $h$, then $\tau_{D/\D}\lambda^h(\tau_{D/\D} G_a)=\tau_{D/\D}G_g$ and $\tau_{D/\D'}\lambda^h(\tau_{D/\D'} G_a)=\tau_{D/\D'}G_g$. But $\tau_{D/\D}\lambda^h=\tau_{D/\D'}\lambda^h$ as $G$ is a $\D'$-algebraic group, and so, since $\tau_{D/\D}G_a=\tau_{D/\D'}G_a$, we get $\tau_{D/\D}G_g=\tau_{D/\D'}G_g$.
\end{proof}

Let $G$ be a $\D$-algebraic group defined over $K$. Then $\tau_{\DD/\D}G$ has naturally the structure of a $\D$-algebraic group defined over $K$; more precisely, if $p:G\times G\to G$ is the group operation on $G$, then $\tau_{\DD/\D}p:\tau_{\DD/\D}G\times \tau_{\DD/\D}G\to \ta G$ is a group operation on $\ta G$. Here we are identifying $\ta(G\times G)$ with $\ta G\times\ta G$ (via the natural isomorphism given in the previous section).

\begin{remark}
By Remark \ref{commu} we have $$\ta p\comp \nabla_\DD^{G\times G}=\nabla_\DD^G\comp p.$$ Hence, the section $\nabla_\DD^G:G\to\ta G$ is a group homomorphism.
\end{remark}

Let us give some useful explicit formulas for the group law of $\ta G$. For each $f\in K\{x\}_\D$ let $$d_\D f_x u:=\sum_{\t\in\T_\D, i\leq n}\frac{\partial f}{\partial \t x_i}(x)\,\t u_i,$$ and for a tuple $f=(f_1,\dots,f_s)$ let $d_\D f_x u$ be $(d_\D(f_1)_x u,\dots,d_\D(f_s)_x u)$. 

If $\lambda^g$ and $\rho^g$ denote left and right multiplication by $g\in G$, respectively, and $p$ is the group operation on $G$, then for $(g,u_1,\dots,u_r)$ and $(h,v_1,\dots,v_r)$ in $\ta G$ we have $$(g,u_1,\dots,u_r)\cdot(h,v_1,\dots,v_r)=(g\cdot h,\,d_\D (\lambda^g)_h v_i+d_\D(\rho^h)_g u_i+p^{D_i}(g,h):i\leq r).$$ 
The inverse is given by 
\begin{displaymath}
(g,u_1,\dots,u_r)^{-1}=(g^{-1},d_\D(\lambda^{g^{-1}}\comp\rho^{g^{-1}})_g(D_i g-u_i)+D_i (g^{-1}):i\leq r).
\end{displaymath}
It is clear, from these formulas, that $\ta G_e$ is a normal $\D$-subgroup of $\ta G$.

\

The following relativization of D-groups is obtained by considering the group objects in the category of relative D-varieties (w.r.t. $\DD/\D$).

\begin{definition}
A \emph{relative  D-group (w.r.t. $\DD/\D$)} defined over $K$ is a relative D-variety $(G,s)$ such that $G$ is a $\D$-algebraic group and $s:G\to \ta G$ is a group homomorphism (all defined over $K$). A \emph{relative D-subgroup} of $(G,s)$ is a relative D-subvariety $(H,s_H)$ of $(G,s)$ such that $H$ is a subgroup of $G$.
\end{definition}

\begin{remark}  \
\begin{enumerate}
\item If $(G,s)$ is a relative D-group then $(G,s)^\#$ is a $\Pi$-subgroup of $G$. Indeed, if $g$ and $h\in (G,s)^\#$ then $$s(g\cdot h^{-1})=s(g)\cdot s(h)^{-1}=(\nabla_\DD g)\cdot (\nabla_\DD h)^{-1}=\nabla_\DD(g\cdot h^{-1}).$$ Hence, $g\cdot h^{-1}\in (G,s)^\#$.
\item Suppose $G$ is a $\D$-algebraic group and $(G,s)$ is a relative D-variety. If $(G,s)^\#$ is a subgroup of $G$ then $(G,s)$ is a relative D-group. Indeed, since $\nabla_\DD^G$ is a group homomorphism, the restriction of $s$ to $(G,s)^\#$ is a group homomorphism, and so, since $(G,s)^\#$ is $\D$-dense in $G$ and $s$ is a regular $\D$-map, $s$ is a group homomorphism on all of $G$.
\end{enumerate}
\end{remark}

The correspondence of Proposition \ref{lema1}~(2) specializes to a natural correspondence between relative D-subgroups of a relative D-group and $\Pi$-subgroups of its set of sharp points:

\begin{lemma}\label{corre}
Let $(G,s)$ be a relative D-group defined over $K$ such that $(G,s)^\#(\bar K)=(G,s)^\#(K)$, for some $\Pi$-closure $\bar K$ of $K$. Then the $\sharp$ operator establishes a 1:1 correspondence between relative D-subgroups of $(G,s)$ defined over $K$ and $\Pi$-algebraic subgroups of $(G,s)^\#$ defined over $K$. The inverse is given by taking $\D$-closures (in the $\D$-Zariski topology) over $K$.
\end{lemma}
\begin{proof}
This is an immediate consequence of Proposition~\ref{lema1}~(2) and the discussion above, except for the fact that the $\D$-closure over $K$ of a $\Pi$-algebraic subgroup of $(G,s)^\#$ is a subgroup of $G$. Let us prove this. Let $H$ be a $\Pi$-algebraic subgroup of $(G,s)^\#$ and $\bar H$ its $\D$-closure over $K$. Let $$X:=\{a\in \bar H:\, b\cdot a\in \bar H \text{ and }b\cdot a^{-1}\in \bar H \text{ for all } b\in \bar H\}\subseteq \bar H ,$$ then $X$ is a $\D$-algebraic subgroup of $G$ defined over $K$. We claim that $\bar H=X$. By the definition of $\bar H$, it suffices to show that $H(\bar K)\subseteq X(\bar K)$. Let $a\in H(\bar K)=H(K)$, and consider $Y_a=\{x\in \bar H:\, x\cdot a\in \bar H \text{ and }x\cdot a^{-1}\in \bar H\}$. Then $H\subset Y_a$, but $Y_a$ is $\D$-closed and defined over $K$, so $\bar H\subset Y_a$. Thus $a\in X$, as desired.
\end{proof}

We end this section by showing that every definable group $G$ of differential type less than the number of derivations is, after possibly replacing $\Pi$ by some independent linear combination, definably isomorphic to a relative D-group (w.r.t. $\DD/\D$) with $|\D|=\Pi$-type$(G)$. By Remark~\ref{witg} we can find a set $\D$ of linearly independent elements of $\operatorname{span}_{K^\Pi}\Pi$ which witnesses the $\Pi$-type of $G$, and hence it suffices to prove the following.

\begin{theorem}\label{co} 
Let $G$ be a connected $K$-definable group. Suppose $\D$ is a set of linearly independent elements of $\operatorname{span}_{K^\Pi}\Pi$ such that $\D$ bounds the $\Pi$-type of $G$ and $|\D|<|\Pi|$. If we extend $\D$ to a basis $\DD\cup\D$ of $\operatorname{span}_{K^\Pi}\Pi$, then $G$ is $K$-definably isomorphic to $(H,s)^{\sharp}$ for some relative D-group $(H,s)$ w.r.t. $\DD/\D$ defined over $K$.
\end{theorem}
\begin{proof}
Recall that any $K$-definable group is $K$-definably isomorphic to a $\Pi$-algebraic group defined over $K$ (to the author's knowledge the proof of this fact, for the partial case, does not appear anywhere; however, it is well known that the proof for the ordinary case \cite{Pi4} extends with little modification). Thus, we assume that $G$ is a $\Pi$-algebraic group defined over $K$. By Corollary \ref{ong}, there is a $\Pi$-rational map $\al$ and a relative D-variety $(V,s)$ w.r.t. $\DD/\D$, both defined over $K$, such that $\al$ yields a $\Pi$-birational equivalence between $G$ and $(V,s)^\#$. Let $p$ be the $\Pi$-generic type of $G$ over $K$ and $q$ be the $\D$-generic type of $V$ over $K$. 

We first show that there is a generically defined $\D$-group structure on $q$. Note that $\al$ maps realisations of $p$ to elements of $(V,s)^\#$ realising $q$. Let $g$ and $h$ be $\Pi$-independent realisations of $p$ then $$K\l g,h\r_\Pi=K\l \al(g),\al(h)\r_\Pi=K\l\al(g),\al(h)\r_\D,$$ and so $\al(g\cdot h)\in K\l \al(g),\al(h)\r_\D$. Thus, we can find a $\D$-rational map $\rho$ defined over $K$ such that $\al(g\cdot h)=\rho(\al(g),\al(h))$. Hence, for any $g,h$ $\Pi$-independent realisations of $p$ $$\al(g\cdot h)=\rho(\al(g),\al(h)),$$ and so $\rho(\al(g),\al(h))$ realises $q$. Now, since $\al(g)$ and $\al(h)$ are $\D$-independent realisations of $q$, we get that for any $x,y$ $\D$-independent realisations of $q$, $\rho(x,y)$ realises $q$. Moreover, if $g,h,l$ are $\Pi$-independent realisations of $p$ then one easily checks that $$\rho(\al(g),\rho(\al(h),\al(l)))=\rho(\rho(\al(g),\al(h)),\al(l)),$$ but $\al(g)$, $\al(h)$ and $\al(l)$ are $\D$-independent realisations of $q$, so for any $x,y,z$ $\D$-independent realisations of $q$ we get $\rho(x,\rho(y,z))=\rho(\rho(x,y),z)$. 

We thus have a stationary type $q$ (in the language of $\D$-rings) and a $\D$-rational map $\rho$ satisfying the conditions of Hrushovski's theorem on groups given generically (see \cite{Hru} or \cite{Po}), and so there is a connected $K$-definable group $H$ and a $K$-definable injection $\beta$ (both in the language of $\D$-rings) such that $\beta$ maps the realisations of $q$ onto the realisations of the generic type of $H$ over $K$ and $\beta(\rho(x,y))=\beta(x)\cdot\beta(y)$ for all $x,y$ $\D$-independent realisations of $q$. We assume, without loss of generality, that $H$ is a $\D$-algebraic group defined over $K$. We have a (partial) definable map $\gamma:=\beta\comp \al:G\to H$ such that if $g,h$ are $\Pi$-independent realisations of $p$ then $\gamma(g\cdot h)=\gamma(g)\cdot \gamma(h)$. It follows that there is a $K$-definable group embedding $\bar \gamma:G\to H$ extending $\gamma$. Indeed, let $$U:=\{x\in G:\gamma(x y)=\gamma(x)\gamma(y)\text{ and } \gamma(yx)=\gamma(y)\gamma(x) \text{ for all }y\models p \text{ with } x\ind_K y\},$$ then $U$ is a $K$-definable subset of $G$ (by definability of types in $DCF_{0,m}$). If $g\models p$ then $g\in U$, and so every element of $G$ is a product of elements of $U$ (see \S 7.2 of \cite{Ma2}). Let $g\in G$ and let $u,v\in U$ be such that $g=u\cdot v$, then $\bar \gamma$ is defined by $$\bar \gamma(g)=\gamma(u)\cdot \gamma(v).$$ 
It is well known that this construction yields a group embedding (see for example \S 3 of Marker's survey \cite{Ma}). We also have a (partial) definable map $t:=\ta \beta\comp s\comp \beta^{-1}:H\to \ta H$ such that for every $g\models p$, as $\al(g)\in (V,s)^\#$, we have $t(\gamma (g))=\nabla_\DD^H(\gamma(g))$. Thus, for $g,h$ $\Pi$-independent realisations of $p$ $$t(\gamma(g)\cdot \gamma(h))=\nabla_\DD^H(\gamma(g)\cdot \gamma(h))=(\nabla_\DD^H\gamma(g))\cdot(\nabla_\DD^H\gamma(h))=t(\gamma(g))\cdot t(\gamma(h)).$$ Thus, for $x,y$ $\D$-independent realisations of the generic type of $H$ over $K$ we have $t(x\cdot y)=t(x)\cdot t(y)$, and so we can extend $t$ to a $K$-definable group homomorphism $\bar t:H\to \ta H$. Therefore, $(H,\bar t)$  is a relative D-group w.r.t. $\DD/\D$. Clearly, if $g$ is a $\Pi$-generic of $G$ over $K$, then $\gamma(g)$ is a $\Pi$-generic (over $K$) of both $\bar{\gamma}(G)$ and $(H,\bar t)^\#$. Hence, $\bar{\gamma}(G)=(H,\bar t)^\#$, as desired.
\end{proof}

\section{Relative logarithmic differential equations}\label{galex}

In this section we introduce the logarithmic derivative on a relative D-group and the Galois extensions associated to logarithmic differential equations. We then show, under some mild assumptions, that these are precisely the generalized strongly normal extensions of Section 1. The theory we develop here extends Pillay's theory from \cite{Pi2}.

We still continue to work in our universal domain $(\U,\Pi)\models DCF_{0,m}$, a base $\Pi$-field $K < \U$ and a partition $\Pi=\DD\cup \D$ with $\DD=\{D_1,\dots,D_r\}$.  

\begin{definition}
Let $(G,s)$ be a relative D-group defined over $K$ and $e$ be its identity. We say that $\al\in \ta G_e$ is an \emph{integrable point of $(G,s)$} if $(G,\al s)$ is a relative D-variety (but not necessarily a relative D-group). Here $\al s:G\to \ta G$ is the $\D$-section given by $(\al s)(x)=\al\cdot s(x)$, where the latter product occurs in $\ta G$. Clearly $\nabla_\DD e\in \ta G_e$ is an integrable point.
\end{definition}

\begin{example}\label{exa}
Suppose $G$ is a linear $\D$-algebraic group defined over $K^{\DD}$ (that is, a $\D$-algebraic subgroup of GL$_n$ for some $n$). Then, by Remark \ref{remi}~(i), $\ta G$ is equal to $(T_\D G)^r$, the $r$-th fibred product of the $\D$-tangent bundle of $G$, and so there is a zero section $s_0:G\to (T_\D G)^r$. Then $(G,s_0)$ is a relative D-group (w.r.t. $\DD/\D$). In this case a point $(\operatorname{Id},A_1,\dots,A_r)\in \ta G_{\operatorname{Id}}=(\mathcal{L}_\D(G))^{r}$, where $\mathcal{L}_\D(G)$ is the $\D$-Lie algebra of $G$, is integrable if and only if it satisfies
\begin{displaymath}
D_iA_j-D_jA_i=[A_i,A_j]
\end{displaymath}
for $i,j=1,\dots,r$. These are the integrability conditions on $A_1,\dots,A_r$ that one finds in \cite{PM} or \cite{VS}.
\end{example}

\begin{lemma}\label{intpt}
A point $\al\in \ta G_e$ is integrable if and only if the system of $\Pi$-equations
\begin{displaymath}
\nabla_{\DD}x=\al s(x)
\end{displaymath}
has a solution in $G$.
\end{lemma}
\begin{proof}
If $\al$ is integrable then, by  Proposition \ref{lema1} (1), $\{g\in G: \nabla_{\DD}g=\al s(g)\}$ is $\D$-dense in $G$. Conversely, suppose there is $g\in G$ such that $\nabla_{\DD}g=\al s(g)$. To prove $\al$ is integrable it suffices to show that $(G,\al s)^{\sharp}$ is $\D$-dense in $G$. This follows from $(G,\al s)^{\sharp}=g(G,s)^{\sharp}$ and the fact that $(G,s)^{\sharp}$ is $\D$-dense in $G$.
\end{proof}

Let $(G,s)$ be a relative D-group defined over $K$. The \emph{logarithmic derivative associated to} $(G,s)$ is defined by
\begin{center}
$\ld:\, G \to \ta G_e$

\qquad \qquad \qquad $g\mapsto (\nabla_{\DD}g)\cdot \, (s(g))^{-1}$
\end{center}
where the product and inverse occur in $\ta G$.

We now list some properties of $\ld$.
\begin{lemma}\label{propder}
\
\begin{enumerate}
\item $\ld$ is a crossed-homomorphism, i.e. $\ld(gh)=(\ld g)(g*\ld h$). Here $*$ is the adjoint action of $G$ on $\ta G_e$, that is, $g*\al:=\ta C_g (\al)$ for each $\al\in \ta G_e$ where $C_g$ denotes conjugation by $g$.
\item The kernel of $\ld$ is $(G,s)^\#$.
\item The image of $\ld$ is exactly the set of integrable points of $(G,s)$.
\item For all $a\in G$, $tp(a/K\, \ld a)$ is $(G,s)^\#$-internal and $a$ is a fundamental system of solutions.
\end{enumerate}
\end{lemma}
\begin{proof} \

\noindent(1) This can be shown as in Pillay \cite{Pi2}. We include a sketch of the proof. An easy computation shows that $g*\al=u\al u^{-1}$ for any $u\in\ta G_g$. Thus, the adjoint action of $G$ on $\ta G_e$ is by automorphisms and
\begin{eqnarray*}
\ld (gh)
&=& (\nabla_{\DD}g)(\nabla_{\DD}h)(s(h))^{-1}(s(g))^{-1} \\
&=& (\ld g) (s(g))(\ld h)(s(g))^{-1} =(\ld g)(g*\ld h).
\end{eqnarray*}
\noindent(2) By definition of $\ld$.

\noindent(3) Follows from Lemma \ref{intpt}, since if $\al=\ld(g)$ for some $g\in G$ then $\nabla_\DD g=\al\, s(g)$.

\noindent(4) By basic properties of crossed-homomorphisms $\ld^{-1}(\ld a)=a\, ker(\ld)=a(G,s)^\#$ for all $a\in G$. Thus, if $b\models tp(a/K\,\ld a)$ then $\ld b=\ld a$, and so $b\in \ld^{-1}(\ld a)= a(G,s)^\#$.
\end{proof}

Extending the work of Pillay in \cite{Pi2}, we point out that these relative logarithmic differential equations give rise to generalized strongly normal extensions.

Let $(G,s)$ be a relative D-group defined over $K$ with $(G,s)^\#(K)=(G,s)^\#(\bar K)$, for some (equivalently any) $\Pi$-closure $\bar K$ of $K$, and $\al$ be an integrable $K$-point of $(G,s)$. By Lemma \ref{propder}~(3), the set of solutions in $G$ to $\ld x=\al$ is nonempty. Hence there is a maximal $\Pi$-ideal $\mathcal M\subset K\{x\}_\Pi$ containing $\{\ld x-\al\}\cup \I(G/K)$. It is a well known fact that every maximal $\Pi$-ideal of $K\{x\}_\Pi$ is a prime ideal (see for example \cite{Ka}). Let $a$ be a tuple of $\U$ such that $\mathcal{M}=\I(a/K)_{\Pi}$. Note that $tp(a/K)$ is therefore isolated (by the formula which sets the radical differential generators of $\I(a/K)_{\Pi}$ to zero) and so $K\l a\r_{\Pi}$ is contained in a $\Pi$-closure of $K$. Moreover, Lemma \ref{propder}~(4) tells us that, $tp(a/K)$ is $(G,s)^\#$-internal and $a$ is a fundamental system of solutions. Hence, by Proposition \ref{inter}, $K\l a\r_{\Pi}$ is a $(G,s)^\#$-strongly normal extension of $K$. We call this the \emph{Galois extension associated to $\ld x=\al$}. The Galois group associated to this generalized strongly normal extension is called the \emph{Galois group associated to $\ld x=\al$}.

Let us point out that the above construction does not depend on the choice of $a$ (up to isomorphism over $K$). Indeed, if $b$ is another solution such that $\I(b/K)_{\Pi}$ is a maximal $\Pi$-ideal, then both $tp(a/K)$ and $tp(b/K)$ are isolated and so we can find a $\Pi$-closure $\bar K$ of $K$ containing $b$ and an embedding $\phi:K\{a\}_{\Pi}\to \bar K$ over $K$. Since $\al$ is a $K$-point, $\ld \phi(a)=\al$, thus, by Lemma \ref{propder}~(1), $b^{-1}\phi(a)\in (G,s)^{\sharp}(\bar K)=(G,s)^{\sharp}(K)$. Hence, $\phi(a)$ and $b$ are interdefinable over $K$ and so $K\l a\r_{\Pi}$ is isomorphic to $K\l b\r_{\Pi}$ over $K$. This argument actually shows that there is exactly one such extension in each $\Pi$-closure of $K$.

\begin{remark}[On the condition $(G,s)^\#(K)=(G,s)^\#(\bar K)$] \
\begin{itemize}
\item [(i)] Let $G$ be a $\D$-group defined over $K^\DD$ and $s_0:G\to \ta G=(T_\D G)^r$ its zero section. If $K^\DD$ is $\D$-closed then $(G,s_0)^\#(K)=(G,s_0)^\#(\bar K)$. On the other hand, if $(G,s_0)^\#(K)=(G,s_0)^\#(\bar K)$ and $\D$-type$(G)=|\D|$ then $K^\DD$ is $\D$-closed.
\item [(ii)] Let $(G,s)$ be a relative D-group defined over $K$. If $(G,s)^\#$ is the Galois group of a generalized strongly normal extension of $K$ then $(G,s)^\#(K)=(G,s)^\#(\bar K)$ (see (6) of Theorem~\ref{group}).
\end{itemize}
\end{remark}

\begin{example}[{\it The linear case}] Suppose $G=GL_n$ and $s_0:G\to \ta G=(T_\D G)^r$ is the zero section. By Proposition \ref{prog}, $T_\D G=TG$ the (algebraic) tangent bundle of $G$, and so $\ta G_{\operatorname {Id}}=\{\operatorname{Id}\}\times (Mat_n)^r$. If $\al=(\operatorname{Id},A_1,\dots,A_r)\in \ta G_{\operatorname {Id}}$, then the logarithmic differential equation $\ell_{s_0} x=\al$ reduces to the system of linear differential equations $$D_1 x=A_1 x,\;\dots\;, D_r x=A_r x.$$ As we already pointed out in Example \ref{exa}, $\al$ will an integrable point if and only if $$D_iA_j-D_jA_i=[A_i,A_j]\; \text{ for } i,j=1,\dots,r.$$ Also, in this case, $(G,s_0)^\#(K)=(G,s_0)^\#(\bar K)$ if and only if $K^\DD$ is $\D$-closed. Thus, the Galois extensions of $K$ associated to logarithmic differential equations of $(GL_n,s_0)$ are precisely the parametrized Picard-Vessiot extensions considered by Cassidy and Singer in \cite{PM}.
\end{example}

\begin{proposition}\label{char}
Let $(G,s)$ be a relative D-group defined over $K$ with $(G,s)^{\sharp}(K)=(G,s)^{\sharp}(\bar K)$ and $\al$ be an integrable $K$-point of $(G,s)$. Let $L$ be a $\Pi$-field extension of $K$ generated by a solution to $\ld x=\al$. Then, $L$ is the Galois extension associated to $\ld x=\al$ if and only if $(G,s)^{\sharp}(K)=(G,s)^{\sharp}(\bar L)$ for some (any) $\Pi$-closure $\bar L$ of $L$.
\end{proposition}
\begin{proof}
By the above discussion, the Galois extension associated to $\ld x=\al$ is a $(G,s)^\#$-strongly normal. Thus, $(G,s)^{\sharp}(K)=(G,s)^{\sharp}(\bar L)$. For the converse, suppose $L=K\l b\r_{\Pi}$ where $b$ is a solution to $\ld x=\al$. Then, since $tp(b/K)$ is $(G,s)^\#$-internal and $b$ is a fundamental system of solutions, $L$ is a $(G,s)^\#$-strongly normal extension and so $L$ is contained in a $\Pi$-closure $\bar K$ of $K$. Let $a$ be a tuple from $\bar K$ such that $\I(a/K)_{\Pi}$ is a maximal $\Pi$-ideal. Then $K\l a\r_{\Pi}$ is the Galois extension associated to $\ld x=\al$. But, by Lemma \ref{propder}, $b^{-1}a\in (G,s)^{\sharp}(\bar K)=(G,s)^\#(K)$, and hence $L=K\l a\r_{\Pi}$. 
\end{proof}

The proof of Lemma 3.9 of \cite{Pi2} extends directly to the partial case and yields the following proposition.

\begin{proposition}\label{gal}
Let $(G,s)$ and $\al$ be as in Proposition \ref{char}. Then the Galois group associated to $\ld x =\al$ is of the form $(H,s_H)^{\sharp}$ for some relative D-subgroup $(H,s_H)$ of $(G,s)$ defined over $K$. Moreover, if $a$ is a solution to $\ld x=\al$ such that $\I(a/K)_{\Pi}$ is a maximal $\Pi$-ideal, then the action of $(H,s_H)^\#$ on $tp(a/K)^{\U}$ is given by $h.b=(aha^{-1})b$.
\end{proposition}
\begin{proof}
We give a slightly more direct argument than what is found in \cite{Pi2}. Let $L=K\l a\r_{\Pi}$ be the Galois extension associated to $\ld x=\al$, where $a$ is a solution to $\ld x=\al$ and $\I(a/K)_{\Pi}$ is a maximal $\Pi$-ideal. Let $Z$ be the $\Pi$-locus of $a$ over $K$, note that $Z=tp(a/K)^{\U}$ and that $Z$ is a $\Pi$-algebraic subvariety of $(G,\al s)^{\sharp}$. Let $f$ be the multiplication on $G$ and $Y=\{g\in (G,s)^{\sharp}:Zg=Z\}$. Following the construction of Theorem \ref{group} (1) with the data $((G,s)^\#,a,Y,f)$, we get a bijection $\mu:\,$Gal$(L/K)_{\Pi}\to Y$ defined by $\mu(\sigma)=a^{-1}\sigma(a)$. It follows that $\mu$ is in fact a group isomorphism, where $Y$ is viewed as a subgroup of $(G,s)^{\sharp}$. Let $H$ be the $\D$-closure of $Y$ over $K$. Since $Y$ is a $\Pi$-algebraic subgroup of $(G,s)^\#$ defined over $K$, $H$ equipped with $s_H:=s|_H$ is a relative D-subgroup of $(G,s)$ defined over $K$ and $(H,s_H)^\#=Y$ (see Lemma \ref{corre}).

The moreover clause follows by (\ref{action1}) in the proof of Theorem \ref{group}~(1).
\end{proof}

The Galois correspondence given by Theorem \ref{normal} specializes to this context and, composed with the bijective correspondence between relative D-subgroups of a given relative D-group and the $\Pi$-algebraic subgroups of the sharp points given in Lemma \ref{corre}, yields the following correspondence.

\begin{corollary}
Let $(G,s)$ and $\al$ be as in Proposition \ref{char}. Let $L$ be the Galois extension associated to $\ld x=\al$ with Galois group $(H,s_H)^\#$. Then there is a Galois correspondence between the intermediate $\Pi$-fields (of $K$ and $L$) and the relative D-subgroups of $(H,s_H)$ defined over $K$.
\end{corollary}

Now, let $(G,s)$ and $\al$ be as in Proposition \ref{char}, and $L$ be the Galois extension associated to $\ld x=\al$ with Galois group $(H,s_H)^\#$. Suppose that there is $\D'\subseteq \D$ such that $G$ is a $\D'$-algebraic group and $s$ is a $\D'$-section (recall that in this case $\ta G=\tau_{\DD/\D'}G$). Let $\Pi'=\DD\cup\D'$. We can consider the Galois extension $L'$ and Galois group $(H',s_{H'})^\#$ associated to $\ld x=\al$ when the latter is viewed as a logarithmic $\Pi'$-equation (note that $\al$ is also an integrable point when $(G,s)$ is viewed as a relative D-group w.r.t. $\DD/\D'$). In other words, $L'$ is a $\Pi'$-field extension of $K$ of the form $K\l a\r_{\Pi'}$ where $\ld a=\al$ and $\I(a/K)_{\Pi'}$ is a maximal $\Pi'$-ideal of $K\{ x\}_{\Pi'}$, and $(H',s_{H'})$ is a relative D-subgroup w.r.t. $\DD/\D'$ of $(G,s)$ such that $(H',s_{H'})^\#$ is (abstractly) isomorphic to the group of $\Pi'$-automorphisms Gal$(L'/K)_{\Pi'}$. We have the following relation between the Galois extensions $L$ and $L'$, and the groups $H$ and $H'$:

\begin{proposition}\label{as}
Let $L$, $L'$, $H$ and $H'$ be as above. If $L=K\l a\r_\Pi$ then $L'=K\l a\r_{\Pi'}$, and $H'$ equals the $\D'$-closure (in the $\D'$-Zariski topology) of $H$ over $K$.
\end{proposition}
\begin{proof}
Since $(G,s)^\#(K)=(G,s)^\#(\bar L)$ for some $\Pi$-closure $\bar L$ of $L$, then $(G,s)^\#(K)=(G,s)^\#(\overline{K\l a\r}_{\Pi'})$ for some $\Pi'$-closure $\overline{K\l a\r}_{\Pi'}$ of $K\l a\r_{\Pi'}$. Now Proposition \ref{char} implies that $K\l a\r_{\Pi'}$ is the Galois extension associated to $\ld x=\al$ when viewed as a logarithmic $\Pi'$-equation, and so $L'=K\l a\r_{\Pi'}$. Now, to show that $H'$ is the $\D'$-closure of $H$ over $K$ it suffices to show that $(H',s_{H'})^\#$ is the $\Pi'$-closure of $(H,s_H)^\#$ over $K$. First we check that $(H,s_H)^\#\subseteq (H',s_{H'})^\#$. Let $h\in (H,s_H)^\#$, then there is $\sigma\in$ Gal$(L/K)_\Pi$ such that $h=a^{-1}\sigma(a)$. From the latter equation it is easy to see that $\sigma(a)\in L'\l(G,s)^\#\r_{\Pi'}$, and so $\sigma$ restricts to an element of Gal$(L'/K)_{\Pi'}$. Thus, $a^{-1}\sigma(a)\in (H',s_{H'})^\#$, showing the desired containment. Let $Y$ be the $\Pi'$-closure of $(H,s_H)^\#$ over $K$, then adapting the proof of Lemma \ref{corre} one can see that $Y$ is a $\Pi'$-algebraic subgroup of $(H',s_{H'})^\#$. The fixed field of $(H,s_H)^\#$ is all of $K\l (G,s)^\#\r_\Pi$, then the fixed field of $Y$ (as a $\Pi'$-subgroup of $(H',s_{H'})^\#$) is all of $K\l(G,s)^\#\r_{\Pi'}$. Hence, by the Galois correspondence, $Y=(H',s_{H'})^\#$.
\end{proof}

In the case when $(G,s)=(GL,s_0)$ and $\D'=\emptyset$, Proposition \ref{as} specializes to Proposition 3.6 of \cite{PM}.

\

We finish this section by showing that, under some natural assumptions on the differential field $K$, every generalized strongly normal extension of $K$ is the Galois extension of a logarithmic differential equation. This is in analogy with Remark 3.8 of \cite{Pi2}.

\begin{theorem}\label{main}
Let $X$ be a $K$-definable set and $L$ an $X$-strongly normal extension of $K$. Suppose $\D$ is a set of linearly independent elements of $\operatorname{span}_{K^\Pi}\Pi$ such that $\D$ bounds the $\Pi$-type of $X$, $|\D|<|\Pi|$ and $K$ is $\D$-closed. If we extend $\D$ to a basis $\DD\cup\D$ of $\operatorname{span}_{K^\Pi}\Pi$, then there is a connected relative D-group $(H,s)$ w.r.t. $\DD/\D$ defined over $K$ and $\al$ an integrable $K$-point of $(H,s)$ such that $L$ is the Galois extension associated to $\ld x=\al$.
\end{theorem}
\begin{proof}
We just need to check that the argument given in Proposition 3.4~(ii) of \cite{Pi} for the finite-dimensional case extends to this setting. Let $\G$ be the Galois group of $L$ over $K$. Note that, by Theorem \ref{group}~(5), $\G$ is connected and that, by Lemma~\ref{onint}~(2) and Theorem \ref{group}~(2), $\D$ also bounds the $\Pi$-type of $\G$. Thus, Theorem \ref{co} implies that $\G$ is of the form $(H,s)^\#$ for some relative D-group $(H,s)$ w.r.t. $\DD/\D$ defined over $K$.

Now, let $b$ be a tuple such that $L=K\l b\r_{\Pi}$ and let $\bar K$ be a $\Pi$-closure of $K$ that contains $L$. Let $\mu$ be the canonical isomorphism from Gal$(L/K)$ to $(H,s)^\#$. We know there is some $K$-definable function $h$ such that $\mu(\sigma)=h(b,\sigma(b))$ for all $\sigma \in$ Gal$(L/K)$. Consider the map $\nu:Aut_{\Pi}(\bar K/K)\to H(\bar K)$ defined by $\nu(\sigma)=h(b,\sigma(b))$. Let $\sigma_i\in Aut_{\Pi}(\bar K/K)$ for $i=1,2$, and denote by $\sigma_i'$ the unique elements of Gal$(L/K)$ such that $\sigma_i'(b)=\sigma_i(b)$. We have 
\begin{eqnarray*}
\nu(\sigma_1\comp \sigma_2)
&=& h(b,\sigma_1\comp\sigma_2(b))=h(b,\sigma_1'\comp\sigma_2'(b)) \\
&=& \mu(\sigma_1'\comp\sigma_2')=\mu(\sigma_1')\, \mu(\sigma_2') \\
&=& h(b,\sigma_1'(b))\, h(b,\sigma_2'(b))=h(b,\sigma_1(b))\, h(b,\sigma_2(b)) \\
&=& \nu(\sigma_1)\, \nu(\sigma_2)=\nu(\sigma_1)\, \sigma_1(\nu(\sigma_2)) \\
\end{eqnarray*}
where the last equality follows from $(H,s)^\#(\bar K)=(H,s)^\#(K)$. In the terminology of (\cite{Ko2}, Chap. 7) or \cite{Pi3} we say that $\nu$ is a definable cocycle from $Aut_{\Pi}(\bar K/K)$ to $H$. Using the fact that $K$ is $\D$-closed, we can extend the argument from Proposition 3.2 of \cite{Pi} to show that every definable cocycle is cohomologous to the trivial cocycle. In other words, we get a tuple $a\in H(\bar K)$ such that $\nu(\sigma)=a^{-1}\,\sigma(a)$ for all $\sigma\in Aut_{\Pi}(\bar K/K)$. 

\noindent {\bf Claim.} $K\l a\r_{\Pi}=K\l b \r_{\Pi}$. \\
Towards a contradiction suppose $a\notin K\l b\r_{\Pi}$. Since $a\in \bar K$ and $\bar K$ is also a $\Pi$-closure of $K\l b\r_{\Pi}$ (see \cite{Pi6}, Chap. 8), we get that $tp(a/K,b)$ is isolated. Thus we can find $c\in \bar K$ realising $tp(a/K,b)$ such that $c\neq a$. Then there is $\sigma\in Aut_{\Pi}(\bar K/K\l b\r_{\Pi})$ such that $\sigma(a)=c$ (see \cite{Pi6}, Chap. 8), but this is impossible since $\sigma$ fixes $b$ iff $\nu(\sigma)=e$ (where $e$ is the identity of $H$) iff $a^{-1}\sigma(a)=e$ iff $\sigma$ fixes $a$. The other containment is analogous. This proves the claim.

By Proposition \ref{char}, all that is left to show is that $\al:=l_s(a)$ is a $K$-point. Let $\sigma\in Aut_{\Pi}(\U/K)$, then $a^{-1}\sigma(a)\in (H,s)^\#$. Thus, $\nabla_{\DD}(a^{-1}\sigma(a))=s(a^{-1}\sigma(a))$ and so
\begin{eqnarray*}
\sigma(\al)
&=& \sigma((\nabla_{\DD}a)(s(a))^{-1})\\
&=& (\nabla_{\DD}\sigma(a))(s(\sigma(a)))^{-1}\\
&=& (\nabla_{\DD}a)(s(a))^{-1} \\
&=& \al.
\end{eqnarray*}
As desired.
\end{proof}

\

\section{Examples}

In this section we give two non-linear examples of Galois groups associated to logarithmic differential equations. Our examples are modeled after Pillay's non-linear example given in \cite{Pi2}.

First we exhibit a finite-dimensional non-linear Galois group in two derivations:

\begin{example}
Let $\Pi=\{\d_t,\d_w\}$. Let $G={\mathbb G}_m\times {\mathbb G}_a$ and $s:G\to \tau_{\Pi/\emptyset}G=(TG)^2$ be the (algebraic) section defined by $s(x,y)=(x,y,xy,0,xy,0)$. Then $(G,s)$ is a relative D-group w.r.t. $\Pi/\emptyset$, and the logarithmic derivative $\ld:G\to (TG)^2_{(1,0)}$ is given by $$\ld(x,y)=\left(1,0,\frac{\d_t x}{x}-y,\d_t y,\frac{\d_w x}{x}-y,\d_w y\right).$$ Thus, $$(G,s)^\#=\{(x,y)\in G:\, \d_t y=\d_w y=0 \text{ and } \d_t x=\d_w x=yx\}.$$ We take the ground $\Pi$-field to be $K:=\mathbb{C}(t,e^{ct},e^{cw}:c\in \mathbb{C})$, where we regard $t$ and $w$ as two complex variables, and the $\Pi$-field extension $L:=K(w,e^{2wt+w^2})$. Then $L$ is contained in a $\Pi$-closure $\bar K= \bar{\mathbb{C}}$ of $K$ and $\mathbb C$. 

We now show that $(G,s)^\#(\bar K)=(G,s)^\#(K)$. Let $(a,b)\in (G,s)^\#(\bar K)$, then $b\in \bar{K}^\Pi=\mathbb C$ and $$\d_t\left(\frac{a}{e^{b(t+w)}}\right)=\d_w\left(\frac{a}{e^{b(t+w)}}\right)=0.$$ Thus $a=ce^{b(t+w)}$ for some $c\in\mathbb C$, and so $a\in K$.

Now, as $L$ is generated by $(e^{2wt+w^2},2w)$ and this pair is a solution to
\begin{displaymath}
\left\{
\begin{array}{c}
\d_t x=yx \\
\d_t y=0 \\
\d_w x=(y+2t)x \\
\d_w y=2
\end{array}
\right. ,
\end{displaymath}
$L$ is the Galois extension associated to $\ld (x,y)=(1,0,0,0,2t,2)$. Also, since the transcendence degree of $L$ over $K$ is $2$ and $(G,s)^\#$ is a connected $\Pi$-algebraic group whose Kolchin polynomial is constant equal to $2$, then the Galois group associated to $\ld (x,y)=(1,0,0,0,2t,2)$ is $(G,s)^\#$.
\end{example}

Suppose $\DD=\{\d_1,\dots,\d_r\}$, $\D=\{\d_{r+1},\dots,\d_m\}$ is a partition of $\Pi=\{\d_1,\dots,\d_m\}$. Note that while Proposition \ref{bab} shows that for every connected relative D-group $(G,s)$ the subgroup $(G,s)^\#$ is the Galois group of a generalized strongly normal extension, it is not known (at least to the author) if $(G,s)^\#$ is the Galois group of a logarithmic differential equation. The following proposition gives a sufficient condition on $G$ that allows a construction of a Galois extension of a logarithmic differential equation on $(G,s)$ with Galois group $(G,s)^\#$.

\begin{proposition}\label{abo}
Let $(G,s)$ be a relative D-group and suppose $G$ is a connected algebraic group. Then there is $\Pi$-field $K$ and an integrable $K$-point $\al$ of $(G,s)$ such that the Galois group associated to $\ld x=\al$ is  $(G,s)^\#$.
\end{proposition}
\begin{proof}
We follow the construction given by Pillay in \S 4 of \cite{Pi}. Let $K_0$ be a $\Pi$-closed field over which the relative D-group $(G,s)$ is defined and let $a$ be a $\Pi$-generic point of $G$ over $K_0$. Then, since $G$ is an algebraic group, $RU(a/K_0)=\omega^m\cdot d$ where $d$ is the (algebraic-geometric) dimension of $G$. Let $\al=\ld a$, $K=K_0\l \al\r_\Pi$ and $L=K\l a\r_\Pi$. Note that $L=K_0\l a\r_\Pi$.

We first check that $(G,s)^\#(\bar L)=(G,s)^\#(K_0)$. Since $K_0$ is $\Pi$-closed, if $b\in \bar L$ and $b\ind_{K_0} L$ then $b\in K_0$ (this follows from definability of types in $DCF_{0,m}$). Thus it suffices to show that $g\ind_{K_0} a$ for all $g\in (G,s)^\#$. Let $g\in (G,s)^\#$, then $RU(g/K_0)<\omega ^m$ and by the Lascar inequalities $$\omega^m \cdot d \leq RU(a,g/K_0)\leq RU(a/K_0,g)\oplus RU(g/K_0).$$ Hence, $RU(a/K_0,g)=\omega^m\cdot d$ and so $g\ind_{K_0}a$.
It follows from Proposition \ref{char} that $L$ is the Galois extension of $K$ associated to $\ld x=\al$. 

Now we check that the Galois group is $(G,s)^\#$. From the proof of Proposition \ref{gal} we see that the Galois group is equal to $$\{g\in (G,s)^\#: a\cdot g=\sigma(a) \text{ for some }\sigma\in \,\text{Gal}(L/K)\},$$ and thus it suffices to show that for each $g\in (G,s)^\#$ we have that $tp(a\cdot g/K)=tp(a/K)$. Let $g\in (G,s)^\#$, then by Lemma \ref{propder}~(1) we have that $\ld(a\cdot g)=\al$. Since $\{x\in G:\, \ld x=\al\}$ is in definable bijection with $(G,s)^\#$ and the latter has Lascar rank less than $\omega^m$, then $RU(a/K)< \omega^m$ and $RU(a\cdot g/K)<\omega^m$. Using Lascar inequalities again we get that $RU(\al/K_0)=\omega^m\cdot d$ and also that $RU(a\cdot g/K_0)=\omega^m\cdot d$. Then $a\cdot g$ is a generic point of $G$ over $K_0$, and so $tp(a\cdot g/K_0)=tp(a/K_0)$. But $\ld(a\cdot g)=\ld(a)=\al$, thus $tp(a\cdot g/K)=tp(a/K)$ as desired.
\end{proof}

We finish with an example of an infinite-dimensional Galois group associated to a non-linear logarithmic differential equation.

\begin{example}
Let $\Pi=\{\d_1,\d_2\}$. Let $G={\mathbb{G}}_m \times {\mathbb{G}}_a$ and $s:G\to \tau_{\d_2/\d_1}G=TG$ be the $\d_1$-section given by $$s(x,y)=(x,y,xy,\d_1 y).$$ Then $(G,s)$ is a relative D-group w.r.t. $\d_2/\d_1$. The logarithmic derivative  $\ld:G\to TG_e$ is given by $$\ld (x,y)=(1,0,\frac{\d_2 x}{x}-y, \d_2 y-\d_1 y).$$ Thus the sharp points are given by $$(G,s)^\#=\{(x,y)\in G: \, \d_2 x=xy \text{ and } \d_2y=\d_1 y\}.$$ Note that $\Pi$-type$(G,s)^\#=1$ and $\Pi$-dim$(G,s)^\#=2$. By Proposition \ref{abo}, there is $(\al_1,\al_2)$ such that $(G,s)^\#$ is the Galois group associated to the non-linear logarithmic differential equation 
\begin{displaymath}
\left\{
\begin{array}{c}
\d_2 x=x(y+\al_1) \\
\d_2 y=\d_1 y+\al_2
\end{array}
\right. .
\end{displaymath}
\end{example}

\


\begin{thebibliography}{10}

\bibitem{PM}
P. Cassidy and M. Singer.
\newblock Galois theory of parametrized differential equations and linear differential algebraic groups.
\newblock Differential Equations and Quantum Groups (IRMA Lectures in Mathematics and Theoretical Physics Vol. 9). EMS Publishing house, pp.113- 157, 2006.

\bibitem{Fre}
J. Freitag.
\newblock Model theory and differential algebraic geometry.
\newblock PhD Thesis. University of Illinois at Chicago, 2012.

\bibitem{Hru}
E. Hrushovski.
\newblock Unidimensional theories are superstable.
\newblock Annals of Pure and Applied Logic 50, pp.117-138, 1990.

\bibitem{Ka}
I. Kaplansky.
\newblock An introduction to differential algebra.
\newblock L'institut de math\'ematique de l'universit\'e de Nancago. Hermann, 1957.

\bibitem{Ko}
E. R. Kolchin.
\newblock Differential algebra and algebraic groups.
\newblock Academic Press. New York, New York 1973.

\bibitem{Ko2}
E. R. Kolchin.
\newblock Differential algebraic groups.
\newblock Academic Press, Inc. 1985.

\bibitem{Ko3}
E. R. Kolchin.
\newblock Algebraic matric groups and the Picard-Vessiot theory of homogeneous linear ordinary differential equations.
\newblock Annals of Mathematics. Vol. 49, No. 1, 1948.

\bibitem{La}
P. Landesman.
\newblock Generalized differential Galois theory.
\newblock Trans. Amer. Math.Soc. 360, pp.4441-4495, 2008.

\bibitem{Le}
O. Le\'on S\'anchez.
\newblock Geometric axioms for differentially closed field with several commuting derivations.
\newblock Journal of Algebra. Vol. 362, pp.107-116, 2012.

\bibitem{Ma}
D. Marker.
\newblock Embedding differential algebraic groups in algebraic groups.
\newblock www.sci.ccny.cuny.edu/~ksda/PostedPapers/Marker050109.pdf

\bibitem{Ma2}
D. Marker.
\newblock Model Theory: An Introduction.
\newblock Springer-Verlag. New York, Inc, 2002. 

\bibitem{Mc}
T. McGrail.
\newblock The model theory of differential fields with finitely many commuting derivations.
\newblock The Journal of Symbolic Logic. Vol. 65, No. 2, pp.885-913, 2000.

\bibitem{Moo}
R. Moosa, A. Pillay and T. Scanlon.
\newblock Differential arcs and regular types in differential fields.
\newblock J. reine angew. Math. 620, pp.35-54, 2008.

\bibitem{Pi}
A. Pillay.
\newblock Differential Galois theory I. 
\newblock Illinois Journal of Mathematics. Vol. 42, N. 4, 1998.

\bibitem{Pi2}
A. Pillay.
\newblock Algebraic D-groups and differential Galois theory. 
\newblock Pacific Journal of Mathematics. Vol. 216, No. 2, 2004.

\bibitem{Pi3}
A. Pillay.
\newblock Remarks on Galois cohomology and definability.
\newblock The Journal of Symbolic Logic. Vol. 62, No. 2, pp. 487-492, 1997.

\bibitem{Pi4}
A. Pillay.
\newblock Some foundational questions concerning differential algebraic groups.
\newblock Pacific Journal of Mathematics. Vol. 179, No.1, 1997.


\bibitem{PiT}
A. Pillay.
\newblock Two remarks on differential fields.
\newblock Model Theory and Applications. Quaderni di Matematica. Volume 11, 2003. Edited by Dipartimento di Matematica, Seconda Universita di Napoli.

\bibitem{Pi6}
A. Pillay.
\newblock An introduction to stability theory.
\newblock Claredon Press. Oxford 1983.


\bibitem{Po}
B. Poizat.
\newblock Stable Groups.
\newblock Mathematical Surveys and Monographs Vol. 87. American Mathematical Society 2001.


\bibitem{VS}
M. van der Put and M. Singer.
\newblock Galois theory of linear differential equations.
\newblock Springer-Verlag Berlin Heidelberg 2003

\end{thebibliography}

\end{document}